\documentclass{amsart}
\usepackage{url}
\usepackage{mathrsfs}
\usepackage{color}
\usepackage{graphicx}
\usepackage{caption}
\usepackage{subcaption}
\usepackage{multirow}
\usepackage{mathtools}
\usepackage{amssymb}

\newcommand{\re}{\mathbb{R}}

\newcommand{\N}{\mathbb{N}}

\newcommand{\eps}{\epsilon}

\newcommand{\dt}{\delta}

\def\af{\alpha}
\def\bt{\beta}

\def\rank{\mbox{rank}}

\newcommand{\Sig}{\Sigma}

\newcommand{\reff}[1]{(\ref{#1})}

\newcommand{\mc}[1]{\mathcal{#1}}
\newcommand{\mt}[1]{\mathtt{#1}}

\newcommand{\supp}[1]{\mbox{supp}(#1)}
\newcommand{\cone}[1]{\mathit{cone}(#1)}
\renewcommand{\vec}[1]{\mathit{vec}(#1)}
\newcommand{\qmod}[1]{\mbox{QM}[#1]}
\newcommand{\st}{\mathit{s.t.}}

\newcommand{\bdes}{\begin{description}}
	\newcommand{\edes}{\end{description}}
\newcommand{\bal}{\begin{align}}
\newcommand{\eal}{\end{align}}
\newcommand{\bnum}{\begin{enumerate}}
	\newcommand{\enum}{\end{enumerate}}
\newcommand{\bit}{\begin{itemize}}
	\newcommand{\eit}{\end{itemize}}
\newcommand{\bea}{\begin{eqnarray}}
\newcommand{\eea}{\end{eqnarray}}
\newcommand{\be}{\begin{equation}}
\newcommand{\ee}{\end{equation}}
\newcommand{\baray}{\begin{array}}
	\newcommand{\earay}{\end{array}}
\newcommand{\bsry}{\begin{subarray}}
	\newcommand{\esry}{\end{subarray}}
\newcommand{\bca}{\begin{cases}}
	\newcommand{\eca}{\end{cases}}
\newcommand{\bcen}{\begin{center}}
	\newcommand{\ecen}{\end{center}}
\newcommand{\bbm}{\begin{bmatrix}}
	\newcommand{\ebm}{\end{bmatrix}}
\newcommand{\bmx}{\begin{matrix}}
	\newcommand{\emx}{\end{matrix}}
\newcommand{\bpm}{\begin{pmatrix}}
	\newcommand{\epm}{\end{pmatrix}}
\newcommand{\btab}{\begin{tabular}}
	\newcommand{\etab}{\end{tabular}}

\newtheorem{theorem}{Theorem}[section]

\theoremstyle{definition}
\newtheorem{example}[theorem]{Example}
\newtheorem{exm}[theorem]{Example}
\newtheorem{alg}[theorem]{Algorithm}

\newtheorem{remark}[theorem]{Remark}

\numberwithin{equation}{section}

\begin{document}

\title[Distributionally Robust Optimization with Moment Ambiguity]
{Distributionally Robust Optimization with Moment Ambiguity Sets}

\author[Jiawang Nie]{Jiawang~Nie}
\address{Jiawang Nie, Department of Mathematics,
University of California San Diego,
9500 Gilman Drive, La Jolla, CA, USA, 92093.}
\email{njw@math.ucsd.edu}

\author[Liu Yang]{Liu~Yang}
%\address{Liu Yang, school of Mathematics and Computational Sciences,
%Xiangtan University, Xiangtan, Hunan, China, 411105.}
%\email{yangl410@xtu.edu.cn}

\author[Suhan Zhong]{Suhan~Zhong}
\address{Suhan Zhong, Department of Mathematics,
Texas A\&M University, College Station, TX, USA, 77843-3368.}
\email{suzhong@tamu.edu}

\author[Guangming Zhou]{Guangming~Zhou}
\address{Liu Yang, Guangming Zhou,
School of Mathematics and Computational Sciences,
Xiangtan University, Xiangtan, Hunan, China, 411105.}
\email{yangl410@xtu.edu.cn, zhougm@xtu.edu.cn}



\begin{abstract}
This paper studies distributionally robust optimization (DRO)
when the ambiguity set
is given by moments for the distributions.
The objective and constraints are given by polynomials in decision variables.
We reformulate the DRO with equivalent moment conic constraints.
Under some general assumptions,
we prove the DRO is equivalent to a linear optimization problem
with moment and psd polynomial cones.
A Moment-SOS relaxation method is proposed to solve it.
Its asymptotic and finite convergence are shown under certain assumptions.
Numerical examples are presented to show how to solve DRO problems.
\end{abstract}

\maketitle

\section{Introduction}

Many decision problems are involved with uncertainties.
People often like to make a decision that works well with uncertain data.
The distributionally robust optimization (DRO)
is a frequently used model for this kind of decision problems.
A typical DRO problem is
\be \label{md:ecDRPO}
\min_{x\in X}\, f(x)\quad \st\quad
\inf_{\mu\in\mathcal{M}}\mathbb{E}_{\mu}[h(x,\xi)]\ge 0,
\ee
where $f:\re^n\rightarrow \re$, $h:\re^n\times \re^p\rightarrow\re$, 
$x \coloneqq (x_1, \ldots, x_n)$ is the decision variable
constrained in a set $X\subseteq\mathbb{R}^n$
and $\xi \coloneqq (\xi_1,\ldots, \xi_p) \in \mathbb{R}^p$ is the random variable
obeying the distribution of a measure $\mu \in \mc{M}$.
The notation $\mathbb{E}_{\mu}[h(x,\xi)]$ stands for the expectation of
the random function $h(x,\xi)$ with respect to the distribution of $\xi$.
The set $\mathcal{M}$ is called the {\it ambiguity set},
which is used to describe the uncertainty of the measure $\mu$.

The ambiguity set $\mathcal{M}$
is often moment-based or discrepancy-based.
For the moment-based ambiguity, the set $\mc{M}$ is usually specified by
the first, second moments \cite{DelageYeDRO,
HanasuantoDRO, WiesemannConvexDRO}.
Recently, higher order moments are also often used \cite{ZChenDROica2019, GaussianMixModelTensor,LasserreMixAmbiguity},
especially in relevant applications with machine learning.
For discrepancy-based ambiguity sets, popular examples are the
$\phi$-divergence ambiguity sets \cite{BenTaiRSO2013, DataDrivenPolyDRO}
and the Wasserstein ambiguity sets \cite{PflugAmbiguity}.
There are also some other types of ambiguity sets.
For instance, \cite{KlerkDROpolydense} assumes $\mathcal{M}$
is given by distributions with sum-of-squares (SOS)
polynomial density functions of known degrees.

We are mostly interested in Borel measures
whose supports and moments, up to a given degree $d$,
are respectively contained in given sets $S \subseteq \re^p$ and $Y\subseteq\mathbb{R}^{\binom{p+d}{d}}$.
%(note $\binom{p+d}{d}$ denotes the combinatorial
%of choosing $d$ elements from $p+d$ ones).
Let $\mathcal{B}(S)$ denote the set of Borel measures supported in $S$.
We assume the ambiguity set is given as
\begin{equation}
\label{eq:mom_ambi}
\mathcal{M} \coloneqq \Big\{\mu\in \mathcal{B}(S):
\mathbb{E}_{\mu}([\xi]_d)\in Y \Big\},
\end{equation}
where $[\xi]_d$ is the monomial vector
\[
[\xi]_d \,:= \, \bbm 1 & \xi_1 & \cdots & \xi_p & \xi_1^2 & \xi_1\xi_2
   & \cdots & \xi_p^d \ebm^T.
\]
The problem (\ref{md:ecDRPO}) equipped with the above ambiguity set
is called the distributionally robust optimization of moment (DROM).
When all the defining functions are polynomials,
the DROM is an important class of distributionally robust optimization.
It has broad applications.
%This is motivated by the duality between nonnegative polynomial and moment cones.
Polynomial and moment optimization are studied extensively \cite{JLasserrePolyMoment,LaurentSOSmom2009,JNieFlatTruncation,JNieLinearOptMoment}.
This paper studies how to solve DROM in the form \reff{md:ecDRPO}
by using Moment and SOS relaxations
(see the preliminary section for a brief review of them).
Currently, there exists relatively few work on this topic.

Solving DROM is of broad interests recently.
It is studied in \cite{KlerkDROpolydense,DataDrivenPolyDRO}
when the density functions are given by polynomials.
Polynomial and moment optimization are studied extensively \cite{JLasserrePolyMoment,LaurentSOSmom2009,JNieFlatTruncation,JNieLinearOptMoment}.
In this paper, we study how to solve DROM in the form \reff{md:ecDRPO}
by using Moment-SOS relaxations.
Currently, there exists relatively less work on this topic.

We remark that the distributionally robust min-max optimization
\be \label{md:minmax}
\min_{x\in X}\,\max_{\mu\in\mathcal{\mathcal{M}}}\mathbb{E}_{\mu}[F(x,\xi)]
\ee
is a special case of the distributionally robust optimization
in the form (\ref{md:ecDRPO}). Assume each $\mu\in\mathcal{M}$ is a probability measure
(i.e., $\mathbb{E}_{\mu}[1] = 1$), then the min-max optimization
\reff{md:minmax} is equivalent to
\be
\min_{(x, x_0) \in X \times \re}\, x_0 \quad \st\quad \inf_{\mu\in\mathcal{M}}\mathbb{E}_{\mu}[x_0 -F(x,\xi)]\ge 0 .
\ee
This is a distributionally robust optimization problem in the form (\ref{md:ecDRPO}).

The distributionally robust optimization is frequently used
to model uncertainties in various applications.
It is closely related to stochastic optimization and robust optimization.
Under certain conditions,
the DRO can be transformed into other two kinds of problems.
In stochastic optimization
(see \cite{IntrotoSP,SAforSPLan,LanBk20,RusSha03,StochasticOpt}),
people often need to solve decision problems
%with full information of the random variables. It usually
such that the true distributions can be well approximated by sampling.
The performance of computed solutions heavily relies on the quality of sampling.
In order to get more stable solutions,
regularization terms can be added to the optimization
(see \cite{NieSPO,SOstability,DRSVI}).
In robust optimization (see\cite{RobustConvexOpt,TheoApplRobustOpt}),
the uncertainty is often assumed to be freely distributed in some sets.
This approach is often computationally tractable and suitable for large-scale data.
However, it may produce too pessimistic decisions for certain applications.
Combining these two approaches may give more reasonable decisions sometimes.
Some information of random variables may be well estimated,
or even exactly generated from the sampling or historic data.
For instance, people may know the support of the measure,
discrepancy from a reference distribution, or its descriptive statistics.
The ambiguity set can be given as a collection of measures satisfying such properties.
It contains some exact information of distributions, as well as some uncertainties.
For decision problems with ambiguity sets, it is naturally to find
optimal decisions that work well under uncertainties.
This gives rise to distributionally robust optimization like (\ref{md:ecDRPO}).
%

%Lots of researchers from different backgrounds make huge contributions to the model.
We refer to \cite{twostageDRO,KlerkDROpolydense,PostekROwSC,
MomentDROXu,JZhangDRO, DRSOzhang,ZymlerDRchanceMom}
for recent work on distributionally robust optimization. %%the model (\ref{md:ecDRPO}).
For the min-max robust optimization (\ref{md:minmax}),
we refer to \cite{DelageYeDRO,ShapiroAhmedminmaxSO,WiesemannConvexDRO}.
The distributionally robust optimization has broad applications, e.g.,
portfolio management \cite{DelageYeDRO,EsfahaniDDdro,SZCVaRRobustPortfolio},
network design \cite{DataDrivenPolyDRO,YYDROnetwork},
inventory problems \cite{DBAdaptiveDRO,WiesemannConvexDRO}
and machine learning \cite{JDLearningModelDRO,GurRusZhu20,AgnosticFedLearning}.
For more general work on distributonally robust optimization,
we refer to the survey \cite{Rahimiansurvey} and the references therein.

\subsection*{Contributions}

This article studies the distributionally robust optimization (\ref{md:ecDRPO})
with a moment ambiguity set $\mc{M}$ as in (\ref{eq:mom_ambi}).
Assume the measure support set $S$ is a semi-algebraic set given by a tuple $g \coloneqq (g_1,\ldots,g_{m_1})$ of polynomials in $\xi$.
Similarly, assume the feasible set $X$
is given by a polynomial tuple $c \coloneqq (c_1,\ldots,c_{m_2})$ in $x$.
We consider the case that
the objective $f(x)$ is a polynomial in $x$, constrained in a set $X \subseteq \re^n$,
and that the function $h(x,\xi)$ is polynomial in the random variable
$\xi$ and is linear in $x$. The function $h(x,\xi)$ can be written as
\be
h(x,\xi) \,  \coloneqq  \sum_{
\substack{ \af  \coloneqq  (\af_1, \ldots, \af_p )  \\  \af_1 + \cdots + \af_p  \le d }  }
h_\af(x) \cdot  \xi_1^{\af_1} \cdots \xi_{p}^{\af_p} ,
\ee
where each coefficient $h_\af(x)$ is a linear function in $x$.
The total degree in $\xi  \coloneqq  (\xi_1, \ldots, \xi_p)$ is at most $d$.
For neatness, we also write that
\be  \label{eq:linear_h}
h(x,\xi) \,= \, (Ax+b)^T[\xi]_d,
\ee
for a given matrix $A$ and vector $b$.
Recall that $\mathcal{M}$ has the expression (\ref{eq:mom_ambi}).
It is clear that the set $\mathcal{M}$
consists of truncated moment sequences (tms)
\[
y \,  \coloneqq  \, (y_\af), \quad \mbox{where} \quad
\af  \coloneqq  (\af_1, \ldots, \af_p ), \, |\af| \coloneqq \af_1 + \cdots + \af_p  \le d ,
\]
such that the moment vector $y = \int [\xi]_d \mt{d} \mu$
is contained in a given set $Y$.
In this paper, we focus on the case that $S$ is compact
%semi-algebraic set given by a tuple
%$g \coloneqq (g_1,\ldots, g_{m_1})$ of polynomials in $\xi$
and that $Y$ is a set whose conic hull $\cone{Y}$
can be represented by linear, second order or semidefinite conic inequalities.
For convenience, define the conic hull of moments
\begin{equation}
K \,  \coloneqq  \, cone( \{ \mathbb{E}_{\mu}([\xi]_d) : \mu\in \mathcal{M} \} ) .
\end{equation}
Note that $K$ can also be expressed with $cone(Y)$;
see (\ref{eq:KwconeY}). The constraint in \reff{md:ecDRPO} is the same as
\[
(Ax+b)^T y \ge 0 \quad \forall \, y \in K.
\]
Let $K^*$ denote the dual cone of $K$,
then the above is equivalent to $Ax+b \in K^*$.
Therefore, the problem (\ref{md:ecDRPO}) can be equivalently reformulated as
\be  \label{eq:DRPOequiv}
\min_{x\in X}\,f(x)\quad \st\quad Ax+b\in K^*.
\ee
The moment constraining cone $K$ and its dual cone $K^*$
are typically difficult to describe computationally.
However, they can be successfully solved by Moment-SOS relaxations
(see \cite{JNieAtruncated,JNieLinearOptMoment}).

A particularly interesting case is that $\xi$ is a univariate random variable,
i.e., $p=1$. For this case, the dual cone $K^*$
can be exactly represented by semidefinite programming constraints.
For instance, if $d=4$, $Y$ is the hypercube $[0,1]^5$ and $S=[a_1,a_2]$,
then $cone(Y)$ is the nonnegative orthant and
the cone $K$ can be expressed by the constraints
\[
\bbm y_0 & y_1 & y_2 \\ y_1 & y_2 & y_3 \\ y_2 & y_3 & y_4 \ebm \succeq 0, \quad
(a_1+a_2) \bbm y_1 & y_2  \\ y_2 & y_3  \ebm \succeq
a_1 a_2 \bbm y_0 & y_1   \\ y_1 & y_2   \ebm  +
\bbm   y_2 & y_3 \\  y_3 & y_4 \ebm,
\]
\[
(y_0,y_1,y_2,y_3,y_4) \, \ge 0.
\]
In the above, $X_1 \succeq X_2$ means that $X_1-X_2$
is a positive semidefinite (psd) matrix.
The dual cone $K^*$ can be given by
semidefinite programming constraints dual to the above.
The proof for such expression is shown in Theorem \ref{thm:UTMPconv}.

For the case that $\xi$ is multi-variate, i.e., $p>1$,
there typically do not exist explicit semidefinite programming representations
for the cone $K$ and its dual cone $K^*$.
However, they can be approximated efficiently by Moment-SOS relaxations
(see \cite{JNieAtruncated,JNieLinearOptMoment}).

This paper studies how to solve the equivalent optimization problem
\reff{eq:DRPOequiv} by Moment-SOS relaxations. In computation,
the cone of $Y$ is usually expressed as a Cartesian product of
linear, second order, or semidefinite conic constraints.
A hierarchy of Moment-SOS relaxations is proposed 
to solve \reff{eq:DRPOequiv} globally,
which is equivalent to the distributionally robust optimization \reff{md:ecDRPO}.
It is worthy to note that our convex relaxations
use both ``moment'' and ``SOS'' relaxation techniques,
which are different from the classic work of polynomial optimization
and DROM problems.
%%(\cite{KlerkDROpolydense,DataDrivenPolyDRO}).
In most prior work, usually one of moment and SOS relaxation
is used, but rarely two are used simultaneously.
Under some general assumptions (e.g., the compactness or archimedeanness),
we prove the asymptotic and finite convergence of the proposed Moment-SOS method.
The property of finite convergence makes
our method very attractive for solving DROM.
To check whether a Moment-SOS relaxation is tight or not, one can solve an
$\mathcal{A}$-truncated moment problem with the method in \cite{JNieAtruncated}.
By doing so, we not only compute the optimal values
and optimizers of \reff{eq:DRPOequiv},
but also obtain a measure $\mu$
that achieves the worst case expectation constraint.
This is a major advantage that most other methods do not own.
In summary, our main contributions are:
\begin{itemize}

\item
We consider the new class of distributionally robust optimization problems
in the form (\ref{md:ecDRPO}),  which are given by
polynomial functions and moment ambiguity sets.
The Moment-SOS relaxation method
is proposed to solve them globally.
It has more attractive properties than prior existing methods.
Numerical examples are given to show the efficiency.

\item
When the objective $f(x)$ and the constraining set $X$
are given by SOS-convex polynomials,
we prove the DROM is equivalent to a linear conic optimization problem.

\item
Under some general assumptions,
we prove the asymptotic and finite convergence of the proposed method.
There is little prior work on finite convergence for solving DROM.
In particular, when the random variable $\xi$ is univariate,
we show that the lowest order Moment-SOS relaxation is
sufficient for solving \reff{eq:DRPOequiv} exactly.

\item We also show how to obtain the measure $\mu^*$
that achieves the worst case expectation constraint.

\end{itemize}

The rest of the paper is organized as follows.
Section~\ref{sec:prelim} reviews some preliminary results
about moment and polynomial optimization.
In Section~\ref{sec:SolveDRPO}, we give an equivalent reformulation of
the distributionally robust optimization,
expressing it as a linear conic optimization problem.
In Section~\ref{sec:Moment-SOS}, we give an algorithm of Moment-SOS relaxations
to solve \reff{eq:DRPOequiv}. Some numerical experiments and applications are given in Section~\ref{sec:numerical}.
Finally, we make some conclusions and discussions in Section~\ref{sec:con}.

\section{Preliminaries}\label{sec:prelim}
\subsection*{Notation}

The symbol $\mathbb{R}$ (resp., $\mathbb{R}_+$,  $\mathbb{N}$)
denotes the set of real numbers
(resp.,  nonnegative real numbers, nonnegative integers).
For $t\in \mathbb{R}$, $\lceil t\rceil$
denotes the smallest integer that is greater or equal to $t$.
For an integer $k>0$, $[k] \coloneqq \{1,\cdots,k\}$.
The symbol $\mathbb{N}^n$ (resp., $\mathbb{R}^n$) stands for the set of
$n$-dimensional vectors with entries in $\mathbb{N}$ (resp., $\mathbb{R}$).
For a vector $v$, we use $\|v\|$ to denote its Euclidean norm.
The superscript $^T$ denotes the transpose of a matric or vector.
For a set $S$, the notation $\mathcal{B}(S)$ denotes the set of
Borel measures whose supports are contained in $S$.
For two sets $S_1,S_2$, the operation
\[ S_1+S_2 \,  \coloneqq  \, \{s_1+s_2: \, s_1\in S_1, \, s_2\in S_2\}
\]
is the Minkowski sum. The symbol $e$ stands for the vector of all ones
and $e_i$ stands for the $i$th standard unit vector, i.e.,
its $i$th entry is $1$ and all other entries are zeros.
We use $I_n$ to denote the $n$-by-$n$ identity matrix.
%
%Given a vector $v=(v_1,v_2,\ldots,v_n)\in\mathbb{R}^n$,
%the Hankel matrix $\mathcal{H}(v) \coloneqq (h_{i,j})$
%is the square matrix such that $h_{i,j}=v_{i+j-1}$.
%For a matrix $A\in\mathbb{R}^{n\times n}$,
%the trace of $A$, denoted by $tr(A)$, is the sum of all diagonal entries.
%
A symmetric matrix $W$ is positive semidefinite (psd) if
$v^TWv\geq 0$ for all $v\in \mathbb{R}^n$.
We write $W\succeq 0$ to mean that $W$ is psd.
The strict inequality $W \succ 0$ means that $W$ is positive definite.
%Let $f(x,z)$ denote a continuously differentiable function,
%we use $\nabla f$ to denote its whole gradient and
%$\nabla_z f$ denotes its partial gradient with respect to $z$.

The symbol $\mathbb{R}[x]  \coloneqq  \mathbb{R}[x_1,\cdots,x_n]$
denotes the ring of polynomials in $x$ with real coefficients,
and $\mathbb{R}[x]_d$ is the subset of $\mathbb{R}[x]$
with polynomials of degrees at most $d$.
For a polynomial $f\in \mathbb{R}[x]$, we use $\deg(f)$ to denote its degree.
For a tuple $f = (f_1,\ldots,f_r)$ of polynomials,
the $\deg(f)$ denotes the highest degree of $f_i$.
For a polynomial $p(x)$,  $\vec{p}$ is the coefficient vector of $p$.
For $\af  \coloneqq  (\af_1, \ldots, \af_n)$ and $x \coloneqq  (x_1, \ldots, x_n)$, we denote that
\[
x^\af \, \coloneqq  \, x_1^{\af_1} \cdots  x_n^{\af_n} , \quad
|\af| \, \coloneqq \, \af_1 + \cdots + \af_n.
\]
For a degree $d$, denote the power set
\[
\N_d^n \, \coloneqq  \, \{ \af \in \N^n: \, |\af| \le d \}.
\]
Let $[x]_d$ denote the vector of all monomials
in $x$ that have degrees at most $d$, i.e.,
\[
[x]_d \, \coloneqq  \, \bbm 1 & x_1 & \cdots & x_n & x_1^2 & x_1x_2
   & \cdots & x_n^d \ebm^T.
\]
The notation $\xi^\af$ and $[\xi]_d$ are similarly defined for
$\xi  \coloneqq  (\xi_1, \ldots, \xi_p)$.
The notation $\mathbb{E}_{\mu}[h(\xi)]$ denotes the expectation
of the random function $h(\xi)$ with respect to $\mu$
for the random variable $\xi$. The Dirac measure,
which is supported at a point $u$, is denoted as $\delta_u$.
% remove the commas

Let $V$ be a vector space over the real field $\mathbb{R}$.
A set $C \subseteq V$ is a cone if $a x\in C$
for all $x\in C$ and $a>0$.
%
%A convex cone is a cone $C$ that is satisfying that
%$\alpha x+\beta y\in C,\,\forall x,y\in C,\,\forall \alpha,\beta>0$.
%
For a set $X\subset V$, we denote its closure by $\overline{X}$
in the Euclidean topology. Its conic hull,
which is the minimum convex cone containing $X$, is denoted as $\cone{X}$.
The dual cone of the set $X$ is
\be \label{def:X*}
X^* \,  \coloneqq  \, \{ \ell \in V^*|\, \ell(x)  \ge 0,\,\forall x\in X\},
\ee
where $V^*$ is the dual space of $V$
(i.e., the space of linear functionals on $V$).
Note that $X^*$ is a closed convex cone for all $X$.
For two nonempty sets $X_1, X_2 \in V$, we have
$(X_1+X_2)^*=X_1^*\cap X_2^*$.
When $X_1+X_2$ is a closed convex cone, we also have
$( X_1^*\cap X_2^* )^* = X_1 + X_2$.

In the following, we review some basics in optimization about
polynomials and moments. We refer to
\cite{MomentSOShierarchy,JLasserrePolyMoment,LasserreBook2009,
LaurentSOSmom2009,JNieOptCondition,NieLoc}
for more details about this topic.

\subsection{SOS and nonnegative polynomials}

A polynomial $f\in\mathbb{R}[x]$ is said to be SOS  if
$ f = f_1^2+\cdots+f_k^2$ for some real polynomials $f_i \in \mathbb{R}[x]$.
We use $\Sigma[x]$ to denote the cone of all SOS polynomials in $x$.
The $d$th degree truncation of the SOS cone $\Sigma[x]$ is
\[
\Sigma[x]_d \, \coloneqq  \, \Sigma[x]\cap\mathbb{R}[x]_d.
\]
It is a closed convex cone for each degree $d$.
For a polynomial $f\in\mathbb{R}[x]$, the membership $f\in\Sigma[x]$
can be checked by solving semidefinite programs
\cite{JLasserrePolyMoment,LaurentSOSmom2009}.
In particular, $f$ is said to be SOS-convex \cite{HN10}
if its Hessian matrix $\nabla^2f(x)$ is SOS, i.e.,
$\nabla^2 f=V(x)^T V(x)$ for a matrix polynomial $V(x)$.
%
%Not all convex polynomials are SOS-convex \cite{AhmadConvexSOSconvex}.
%

In this paper, we also need to work with polynomials in $\xi  \coloneqq  (\xi_1, \ldots, \xi_p)$.
For a tuple $g \coloneqq (g_1,\ldots,g_{m_1})$ of polynomials in $\xi$,
its quadratic module is the set
\[
\qmod{g} \,  \coloneqq  \,  \Sigma[\xi]+g_1\cdot\Sigma[\xi]+\cdots+g_{m_1} \cdot\Sigma[\xi] .
\]
The $d$th degree truncation of $\qmod{g}$ is
\[
\qmod{g}_{d} \, \coloneqq  \, \Sigma[\xi]_{d}+g_1\cdot\Sigma[\xi]_{d-\deg(g_1)}+\cdots+
g_{m_1} \cdot\Sigma[\xi]_{d-\deg(g_{m_1})}.
\]
Let $S =\{ \xi \in \re^p : g(\xi)\ge 0\}$ be the set determined by $g$
and let $\mathscr{P}(S)$ denote the set of polynomials that are nonnegative on $S$.
We also frequently use the $d$th degree truncation
\[
\mathscr{P}_d(S) \,  \coloneqq  \, \mathscr{P}(S) \cap \re[\xi]_d.
\]
Then it holds that for all degree $d$
\[
\qmod{g}_{d}  \, \subseteq \, \mathscr{P}_d(S).
\]
The quadratic module $\qmod{g}$ is said to be {\it archimedean}
if there exists a polynomial $\phi \in \qmod{g}$ such that
$\{\xi \in \re^p: \phi(\xi) \ge 0\}$ is compact.
If $\qmod{g}$ is archimedean, then $S$ must be a compact set.
The converse is not necessarily true. However, for compact $S$,
the quadratic module $\qmod{\tilde{g}}$ is archimedean
if $g$ is replaced by $\tilde{g}  \coloneqq  (g, N-\| \xi \|^2)$ for $N$ sufficiently large.
When $\qmod{g}$ is archimedean, if a polynomial $h>0$ on $S$,
then we have $h \in \qmod{g}$ (see \cite{PutinarPositive}).
Furthermore, under some classical optimality conditions,
we have $h \in \qmod{g}$ if $h \ge 0$ on $S$ (see \cite{JNieOptCondition}).

\subsection{Truncated Moment Problems}
\label{ssc:LoptMom}

For the variable $\xi\in\mathbb{R}^p$,
the space of truncated multi-sequences (tms) of degree $d$  is
\[
\mathbb{R}^{\N_d^p}  \coloneqq  \big \{ z = (z_{\alpha})_{\af \in\N_d^p}:
   z_{\alpha}\in\mathbb{R} \big\}.
\]
Each $z \in\mathbb{R}^{\N_d^p}$ determines the linear Riesz functional
$\mathscr{L}_z$ on $\mathbb{R}[\xi]_d$ such that
\begin{equation}
\mathscr{L}_z \Big(\sum_{\alpha\in\N_d^p}
h_{\alpha} \xi^{\alpha} \Big) \, \coloneqq  \,
\sum_{\alpha \in \N_d^p} h_{\alpha} z_{\alpha}.
\end{equation}
For convenience of notation, we also write that
\be \label{def:<q,z>}
\langle q, z\rangle \,  \coloneqq  \,
\mathscr{L}_z(q),\quad q \in \mathbb{R}[\xi]_d.
\ee
For a polynomial $q \in \re[\xi]_{2d}$ and a tms
$z \in\mathbb{R}^{\N_{2k}^p}$, with $k \ge d$,
the $k$th order {\it localizing matrix}  $L_q^{(d)}[z]$ is such that
\be \label{def:localizing_matrix}
\vec{a}^T\left(L_q^{(k)}[z]\right) \vec{b} = \mathscr{L}_z(qab)
\ee
for all $a,b\in\mathbb{R}[\xi]_s$, where $s = k-\lceil \deg(q)/2\rceil$.
In particular, for $q=1$ (the constant one polynomial),
the $L_1^{(k)}[z]$ becomes the so-called {\it moment matrix}
\begin{equation}
\label{def:moment_matrix}
M_k[z] \, \coloneqq  \, L_1^{(k)}[z].
\end{equation}

We can use the moment matrix and localizing matrices to describe dual cones
of quadratic modules. For a polynomial tuple $g=(g_1,\ldots, g_{m_1})$
with $\deg(g) \le 2k$, define the tms cone
\begin{equation}
\label{def:Jg_2d}
\mathscr{S}[g]_{2k} \,  \coloneqq  \left\{ z \in \re^{ \N_{2k}^p } : \,
M_k[z]\succeq 0,\, L_{g_1}^{(k)}[z]\succeq 0, \ldots,
L_{ g_{m_1} }^{(k)}[z]\succeq 0
\right\}.
\end{equation}
It can be verified that (see \cite{JNieLinearOptMoment})
\begin{equation}
\label{eq:dualQMnS}
(\qmod{g}_{2k})^*=\mathscr{S}[g]_{2k} .
\end{equation}
A tms $z = (z_\af) \in \re^{\N_d^p}$ is said to admit a representing measure $\mu$
supported in a set $S \subseteq \re^p$ if
$z_{\alpha} = \int  \xi^{\alpha} \mathtt{d} \mu$ for all
$\af \in \N_d^p$. Such a measure $\mu$ is called an
$S$-representing measure for $z$.
In particular, if $z=0$ is the zero tms, then it admits
the identically zero measure.
Denote by $meas(z,S)$ the set of $S$-measures admitted by $z$.
This gives the moment cone
\be \label{def:moment_cone}
\mathscr{R}_d(S) \, \coloneqq  \, \{z \in\mathbb{R}^{\mathbb{N}_d^p} \mid
meas(z,S)\not=\emptyset\}.
\ee
It is interesting to note that $\mathscr{R}_d(S)$
can also be written as the conic hull
\be
\mathscr{R}_d(S) \,= \, \cone{ \{[\xi]_d:  \xi \in S \} }.
\ee
Recall that $\mathscr{P}_d(S)$ denotes the cone of polynomials
in $\re[\xi]_d$ that are nonnegative on $S$.
It is a closed and convex cone.
For all $h \in\mathscr{P}_d(S)$ and $z\in\mathscr{R}_d(S)$,
it holds that for every $\mu \in meas(z,S)$,
\[
\langle h, z\rangle=\sum_{\alpha\in\mathbb{N}_d^p}
h_{\alpha} z_{\alpha} \, = \, \int h(\xi ) \mt{d}\mu \ge 0.
\]
This implies that $\mathscr{R}_d(S)^* = \mathscr{P}_d(S)$.
When $S$ is compact, we also have  $\mathcal{P}_d(S)^* = \mathscr{R}_d(S)$.
If $S$ is not compact, then
\be \label{dual:MP}
%%\mathscr{R}_d(T)^*=\mathscr{P}_d(T),\quad
\mathscr{P}_d(S)^* \,= \, \overline{ \mathscr{R}_d(S) }.
\ee
We refer to %%Tchakaloff \cite{Tchakaloff},
\cite[Section 5.2]{LaurentSOSmom2009} and \cite{JNieLinearOptMoment} for this fact.
%
%In particular, when $n=1$ and $T=[a,b]$ is an interval,
%$\mathscr{P}_d([a,b])$ reduces to be the truncated quadratic module.
%And so $\mathscr{R}_d([a,b])=\mathscr{S}_d([a,b])$ has a semidefinite characterization.
%

A frequent case is that $S = \{ \xi: g(\xi) \ge 0\}$
is determined by a polynomial tuple $g= (g_1,\ldots, g_{m_1})$.
For an integer $k\ge \deg(g)/2$, a tms $z\in\mathbb{R}^{\mathbb{N}_{2k}^p}$
admits an $S$-representing measure $\mu$ if
$z\in\mathscr{S}[g]_{2k}$ and
\be  \label{def:flat_extension}
\rank\,M_{k-d_0}[z]  \,= \, \rank\,M_k[z],
\ee
where $d_0 = \lceil \deg(g)/2\rceil$.
Moreover, the measure $\mu$ is unique and is $r$-atomic,
i.e., $|\supp{\mu}| = r$, where $r =  \rank\,M_k[z]$.
The above rank condition is called {\it flat extension} or {\it flat truncation}
\cite{CurtoUTMP,JNieFlatTruncation}.
When it holds, the tms $z$ is said to be a {\it flat} tms.
When $z$ is flat, one can obtain the unique representing measure $\mu$
for $z$ by computing Schur decompositions and eigenvalues
(see \cite{ExtractSolGloptipoly}).

To obtain a representing measure for a tms $y\in\re^{ \N_d^p }$ that is not flat,
a semidefinite relaxation method is proposed in \cite{JNieAtruncated}.
Suppose $S$ is compact and the quadratic module $\qmod{g}$ is archimedean.
Select a generic polynomial $R \in \Sig[\xi]_{2k}$,
with $2k > \deg(g)$, and then solve the moment optimization
\be  \label{eq:ATMPrelx}
\left\{ \baray{cl}
\min\limits_{\omega}  & \langle R, \omega\rangle \\
\st &  \omega|_d = y, \,   \omega \in \mathscr{S}[g]_{2k}.
\earay \right.
\ee
In the above $\omega|_d$ denotes the $d$th degree truncation of $\omega$, i.e.,
\be \label{tmstrun:w|d}
\omega|_d \,  \coloneqq  \, (\omega_\af)_{ |\af| \le d }.
\ee
As $k$ increases, by solving \reff{eq:ATMPrelx},
one can either get a flat extension of $y$, or a certificate that
$y$ does not have any representing measure.
We refer to \cite{JNieAtruncated}
for more details about solving truncated moment problems.

\section{Moment Optimization Reformulation}
\label{sec:SolveDRPO}

In this section, we reformulate the distributionally robust optimization
equivalently as a linear conic optimization problem with moment constraints.
We consider the DROM problem
\begin{equation}
\label{md:polyDRPO}
\left\{
\begin{array}{rl}
\min\limits_{x\in\mathbb{R}^n} & f(x)\\
\st & \inf\limits_{\mu\in\mathcal{M}}\mathbb{E}_{\mu}[h(x,\xi)]\ge 0,\\
& x\in X,
\end{array}
\right.
\end{equation}
where $x$ is the decision variable constrained in a set $X \subseteq \re^n$
and $\xi \in \re^p$ is the random variable obeying the distribution of
the measure $\mu$ that belongs to the moment ambiguity set $\mc{M}$.
We assume that the objective $f(x)$ is a polynomial in $x$ and
$h(x,\xi)$ is a polynomial in $\xi$
whose coefficients are linear in $x$.
Equivalently, one can write that
\be \label{def:h(x)}
h(x,\xi)=(Ax+b)^T[\xi]_d,\quad A\in\mathbb{R}^{\binom{p+d}{d}\times n},\, b\in\mathbb{R}^{\binom{p+d}{d}}.
\ee
Suppose measures in the ambiguity set $\mathcal{M}$
have supports contained in the set
\be \label{def:set:S}
S = \{ \xi\in\mathbb{R}^p :
g_1(\xi)\ge 0,\ldots,g_{m_1}(\xi)\ge 0\},
\ee
for a given tuple $g \coloneqq (g_1,\ldots,g_{m_1})$ of polynomials in $\xi$.
The ambiguity set $\mathcal{M}$ can be expressed as
\be  \label{def:M}
\mathcal{M} \coloneqq \left\{\mu\in\mathcal{B}(S)
\left|\,
\mathbb{E}_{\mu}([\xi]_d)\in Y
\right.\right\},
\ee
where $Y$ is the constraining set for moments of $\mu$.
The set $Y$ is not necessarily closed or convex.
The closure of its conic hull is denoted as $\overline{ \cone{Y} }$.
In computation, it is often a Cartesian product of
linear, second order or semidefinite cones. 
The constraining set $X$ for $x$ is assumed to be the set
\be  \label{setX:fi(x)>=0}
X \,  \coloneqq  \, \{ x \in \re^n \mid
c_1(x) \ge 0, \ldots, c_{m_2}(x) \ge 0 \} ,
\ee
for a tuple $c=(c_1,\ldots, c_{m_2})$ of polynomials in $x$.

The DROM (\ref{md:polyDRPO}) can be equivalently reformulated
as polynomial optimization with moment conic conditions.
Observe that
\[
\inf\limits_{\mu\in\mathcal{M}}\mathbb{E}_{\mu}[h(x,\xi)] \ge 0
\, \Longleftrightarrow \,
(Ax+b)^T y \ge 0 ,\,\forall\, y\in\mathscr{R}_d(S) \cap \cone{Y} .
\]
The set $\mathscr{R}_d(S)$ is the moment cone defined as in (\ref{def:moment_cone}).
It consists of degree-$d$ tms' admitting $S$-measures.
For convenience, we denote the intersection
\be
\label{eq:KwconeY}
K \, = \, \mathscr{R}_d(S) \cap \cone{Y} .
\ee
Therefore, we get that
\be  \label{equiv:WCECnDC}
\inf_{\mu\in\mathcal{M}}\mathbb{E}_{\mu}[h(x,\xi)]\ge 0
\,\Longleftrightarrow\,Ax+b\in K^*,
\ee
where $K^*$ denotes the dual cone of $K$.
In view of \reff{def:X*} and \reff{def:<q,z>},
the dual cone $Y^*$ is the following polynomial cone
\be
Y^* = \{\phi \in \re[\xi]_d: \langle \phi, z \rangle
\ge 0, \, \forall \, z \in Y \}.
\ee
Observe the dual cone relations
\[
\mathscr{R}_d(S)^* = \mathscr{P}_d(S), \quad
\mathscr{P}_d(S)^* = \overline{  \mathscr{R}_d(S) } ,
\]
\[
\Big( \mathscr{P}_d(S) + Y^* \Big)^*\,= \,
 \overline{  \mathscr{R}_d(S) }  \cap  \overline{ \cone{Y} }.
\]
When both $\mathscr{R}_d(S)$ and $\cone{Y}$ are closed, we have
\be \label{closure=cap}
\overline{  \mathscr{R}_d(S) \cap \cone{Y} }  \quad = \quad
\overline{  \mathscr{R}_d(S) } \cap \overline{ \cone{Y} } .
\ee
If one of them is not closed, the above may or may not be true.
Note that $C^{**} = C$ if $C$ is a closed convex cone.
When \reff{closure=cap} holds and the sum $\mathscr{P}_d(S)+Y^*$
is a closed cone, we can express the dual cone $K^*$ as
\be  \label{def:K_y^*}
K^* \, = \, \mathscr{P}_d(S) + Y^*.
\ee
As shown in \cite[Proposition B.2.7]{LecModConOpt},
the above equality holds if $\mathscr{R}_d(S), cone(Y)$ are closed and
their interiors have non-empty intersection.
Such conditions are often satisfied for most applications.
Recall that $h(x,\xi) = (Ax+b)^T [\xi]_d$.
The membership $Ax+b \in K^*$ means that $h(x,\xi) \in K^*$.
Therefore, we get the following result.

\begin{theorem} \label{th:DRO<=>Cone}
Assume the set $X$ is given as in \reff{setX:fi(x)>=0}.
If the equality \reff{def:K_y^*} holds, then (\ref{md:polyDRPO})
is equivalent to the following optimization
\be  \label{minf0:Ax+b:in:K*}
%%(P)\quad
\left\{
\baray{cl}
\min\limits_{x\in\mathbb{R}^n}\quad & f(x)\\
\st & c_1(x) \ge 0, \ldots, c_{m_2}(x) \ge 0, \\
 & h(x,\xi) \in \mathscr{P}_d(S)+Y^* .
\earay
\right.
\ee
\end{theorem}

The membership constraint in \reff{minf0:Ax+b:in:K*}
means that $h(x,\xi)$, as a polynomial in $\xi$,
is the sum of a polynomial in $\mathscr{P}_d(S)$ and a polynomial in $Y^*$.
When $f, c_1, \ldots, c_{m_2}$ are all linear functions,
\reff{minf0:Ax+b:in:K*} is a linear conic optimization problem.
When $f$ and every $c_i$ are polynomials,
we can apply Moment-SOS relaxations to solve it.

Recall that $X$ is the set given as in \reff{setX:fi(x)>=0}.
Denote the degree
\[
d_1 \, \coloneqq  \, \max \{ \deg(f)/2, \lceil \deg(c)/2 \rceil \}.
\]
Observe that for all $x\in X$ and $w = [x]_{2d_1}$, it holds that
\[
\begin{array}{l}
\langle f, w \rangle = f(x), \,
M_{d_1}[w] \succeq 0, \\
L_{c_i}^{(d_1)}[w] \succeq 0, \,  i = 1,\ldots, m_2.
\end{array}
\]
We refer to the Subsection~\ref{ssc:LoptMom} for the above notation.
For convenience, define the projection map
$\pi : \re^{ \N^n_{2d_1} } \, \to \, \re^n$ such that
\be \label{map:pi}
\pi(w) \, \coloneqq  \, (w_{e_1}, \ldots, w_{e_n}),
\quad w \, \in \re^{ \N^n_{2d_1} }.
\ee
So, the optimization \reff{minf0:Ax+b:in:K*} can be relaxed to
\be  \label{eq:SOS:inirel}
%%(P_0)\quad
\left\{
\baray{cl}
\min\limits_{(x,w)}\quad & \langle f,w\rangle\\
\st\quad & M_{d_1}[w]\succeq 0, \,
     L_{c_i}^{(d_1)}[w] \succeq 0 \, (i \in [m_2]),   \\
& h(x, \xi) \in\mathscr{P}_d(S)+Y^*,\\
& w_0=1,  x = \pi(w), \, w \in \re^{ \N_{2d_1}^n } .
\earay
\right.
\ee
The relaxation (\ref{eq:SOS:inirel}) is said to be {\it tight}
if it has the same optimal value as \reff{minf0:Ax+b:in:K*} does.
Under the SOS-convexity assumption, the relaxation
(\ref{eq:SOS:inirel}) is equivalent to \reff{minf0:Ax+b:in:K*}.
This is the following result.

\begin{theorem} \label{thm:SOS:rel}
Suppose the ambiguity set $\mc{M}$ is given as in \reff{def:M}
and the set $X$ is given as in \reff{setX:fi(x)>=0}. Assume
the polynomials $f, -c_1, \ldots, -c_{m_2}$ are SOS-convex.
Then, the optimization problems \reff{eq:SOS:inirel} and
\reff{minf0:Ax+b:in:K*} are equivalent in the following sense:
they have the same optimal value, and
$w^*$ is a minimizer of \reff{eq:SOS:inirel}
if and only if $x^*  \coloneqq  \pi(w^*)$ is a minimizer of \reff{minf0:Ax+b:in:K*}.
\end{theorem}
\begin{proof}
Let $w$ be a feasible point for (\ref{eq:SOS:inirel})
and $x = \pi(w)$, then $Ax + b \in K^*$.
Since $f, -c_1, \ldots, -c_{m_2}$ are SOS-convex,
by the Jensen's inequality (see \cite{LasserreConvexity}),
we have the inequalities
\[
f(x) = f(\pi(w)) \le \langle f, w \rangle,
\]
\[
c_i(x) = c_i(\pi(w)) \ge \langle c_i, w \rangle , \, i=1,\ldots, m_2.
\]
The $(1,1)$-entry of $L_{c_i}^{(d_1)}[w]$
is $\langle c_i, w \rangle$, so $L_{c_i}^{(d_1)}[w] \succeq 0$ implies that
$\langle c_i, w \rangle \ge 0$.
This means that $x = \pi(w) \in X$
for every $w$ that is feasible for \reff{eq:SOS:inirel}.
Let $f_0, f_1$ denote the optimal values of
\reff{minf0:Ax+b:in:K*}, \reff{eq:SOS:inirel} respectively.
Since the latter is a relaxation of the former,
it is clearly that $f_0 \ge f_1$. For every $\eps >0$, there exists a feasible $w$
such that $\langle f, w \rangle \le f_1 + \eps$,
which implies that $f(\pi(w)) \le  f_1 + \eps$.
Hence $f_0 \le f_1 + \eps$ for every $\eps > 0$.
Therefore, $f_0 = f_1$, i.e., \reff{eq:SOS:inirel} and \reff{minf0:Ax+b:in:K*}
have the same optimal value.

If $w^*$ is a minimizer of \reff{eq:SOS:inirel},
we also have $x^* = \pi(w^*) \in X$ and
\[
f(x^*) = f(\pi(w^*)) \le \langle f, w^* \rangle.
\]
Since \reff{eq:SOS:inirel} is a relaxation of \reff{minf0:Ax+b:in:K*},
they must have the same optimal value
and $x^*$ is a minimizer of \reff{minf0:Ax+b:in:K*}.
For the converse, if $x^*$ is a minimizer of \reff{minf0:Ax+b:in:K*},
then $w^*  \coloneqq  [x^*]_{2d_1}$ is feasible for \reff{eq:SOS:inirel} and
$f(x^*) = \langle f, w^* \rangle$ .
So $w^*$ must also be a minimizer of  \reff{eq:SOS:inirel},
since \reff{eq:SOS:inirel} and \reff{minf0:Ax+b:in:K*}
have the same optimal value.
\end{proof}

In the following, we derive the dual optimization of (\ref{eq:SOS:inirel}).
As in Subsection~\ref{ssc:LoptMom}, we have seen that
\[
M_{d_1}[w] \succeq 0, \,\, L_{c_i}^{(d_1)}[w]\succeq 0 \, (i \in [m_2])
\, \Longleftrightarrow\,  w\in \mathscr{S}[c]_{2d_1},
\]
where $\mathscr{S}[c]_{2d_1}$ is given similarly as in \reff{def:Jg_2d}.
Recall the dual relationship
$(\qmod{c}_{2d_1})^* = \mathscr{S}[c]_{2d_1}$,
as shown in (\ref{eq:dualQMnS}).
The Lagrange function for (\ref{eq:SOS:inirel}) is
\[
\baray{rcl}
 \mathcal{L}(w;\gamma,q,y,z)
& = &\langle f, w\rangle-\gamma(w_0-1)-\langle q,w\rangle-
     \langle y,A \pi(w) +b\rangle \\
& = & \langle f -q-y^TA x-\gamma \cdot 1, w\rangle +
         \gamma-\langle b,y\rangle ,
\earay
\]
for $\gamma \in \re, q\in \qmod{c}_{2d_1}, y \in \overline{K}$.
(Note that the cone $K$ is not necessarily closed.)
To make $\mathcal{L}(w;\gamma,q,y,z)$ have a finite infimum for
$w \in \re^{ \N_{2d_1}^n }$, we need the constraint
\[
  f- y^TA x-\gamma \,=\, q.
\]
Therefore, the dual optimization of (\ref{eq:SOS:inirel}) is
\be \label{eq:SOS:dual}
%%(D_0)\quad
\left\{
\baray{cl}
\max\limits_{(\gamma,y)} & \gamma-\langle b,y\rangle \\
\st  & f(x) - y^TAx-\gamma\in \qmod{c}_{2d_1}, \\
     & \gamma \in \re, \, y \in \overline{K}.
\earay
\right.
\ee
The first membership in \reff{eq:SOS:dual} means that
$f(x) - y^TAx-\gamma$, as a polynomial in $x$,
belongs to the truncated quadratic module $\qmod{c}_{2d_1}$.
So it gives a constraint for both $\gamma$ and $y$.

\section{The Moment-SOS relaxation method}
\label{sec:Moment-SOS}

In this section, we give a Moment-SOS relaxation method for solving
the distributionally robust optimization and prove its convergence.

%
%\subsection{Moment-SOS relaxation}
%
In Section~\ref{sec:SolveDRPO}, we have seen that the DROM (\ref{md:polyDRPO})
is equivalent to the linear conic optimization (\ref{eq:SOS:inirel})
under certain assumptions. It is still hard to solve (\ref{eq:SOS:inirel})
directly, due to the membership constraint $h(x,\xi) \in \mathscr{P}_d(S)+Y^*$.
This is because the nonnegative polynomial cone $\mathscr{P}_d(S)$
typically does not have an explicit computational representation.
For its dual problem (\ref{eq:SOS:dual}), it is similarly difficult
to deal with the conic membership $y \in \overline{K}$.
However, both  (\ref{eq:SOS:inirel}) and (\ref{eq:SOS:dual})
can be solved efficiently by Moment-SOS relaxations.

Recall that $S$ is a semi-algebraic set given as in \reff{def:set:S}.
For every integer $k \ge d/2$, it holds the nesting containment
\[
\qmod{g}_{2k} \cap \re[\xi]_d \subseteq
\qmod{g}_{2k+2} \cap \re[\xi]_d \subseteq \cdots\subseteq \mathscr{P}_d(S).
\]
We thus consider the following restriction of (\ref{eq:SOS:inirel}):
\be  \label{eq:SOS:Qres}
%%(P_k)\quad
\left\{
\baray{cl}
\min\limits_{(x,w)}  & \langle f,w\rangle\\
\st  &
M_{d_1}[w]\succeq 0, \, L_{c_i}^{(d_1)}[w] \succeq 0 \, (i \in [m_2]),   \\
& h(x,\xi) \in \qmod{g}_{2k}+Y^*,\\
& w_0 = 1, x = \pi(w), w \in \re^{ \N_{2d_1}^n } .
\earay
\right.
\ee
The integer $k$ is called the relaxation order.
Since $(\qmod{g}_{2k})^*=\mathscr{S}[g]_{2k}$,
the dual optimization of (\ref{eq:SOS:Qres}) is
\be  \label{eq:SOS:Qdual}
%%(D_k)\quad
\left\{
\baray{cl}
\max\limits_{(\gamma,y,z)}  & \gamma-\langle b,y\rangle \\
\st  & f(x) -y^TAx - \gamma \in  \qmod{c}_{2d_1},\\
& \gamma \in \re, \, z\in \mathscr{S}[g]_{2k}, \,
y \in \overline{ \cone{Y} },\,  y = z|_d .
\earay
\right.
\ee
We would like to remark that $\qmod{g}$
is a quadratic module in the polynomial ring $\re[\xi]$,
while $\qmod{c}$ is a quadratic module in the polynomial $\re[x]$.
The notation $z|_d$ denotes the degree-$d$ truncation of $z$;
see \reff{tmstrun:w|d} for its meaning.
The optimization (\ref{eq:SOS:Qdual}) is a relaxation of (\ref{eq:SOS:dual}),
since it has a bigger feasible set.
There exist both quadratic module and moment
constraints in (\ref{eq:SOS:Qdual}).
The primal-dual pair \reff{eq:SOS:Qres}-\reff{eq:SOS:Qdual}
can be solved as semidefinite programs.
The following is a basic property about the above optimization.

\begin{theorem} \label{thm:TightRel}
Assume \reff{closure=cap} holds.
Suppose $(\gamma^*,y^*,z^*)$ is an optimizer of (\ref{eq:SOS:Qdual})
for the relaxation order $k$.
Then $(\gamma^*,y^*)$ is a maximizer of (\ref{eq:SOS:dual})
if and only if it holds that
$
 y^* \in \overline{ \mathscr{R}_d(S)   } .
$
\end{theorem}
\begin{proof}
If $(\gamma^*,y^*)$ is a maximizer of (\ref{eq:SOS:dual}),
then it is clear that $y^* \in \overline{ \mathscr{R}_d(S) }$.
Conversely, if $y^* \in \overline{ \mathscr{R}_d(S) }$, then $(\gamma^*, y^*)$
is feasible for \reff{eq:SOS:dual}, since \reff{closure=cap} holds.
Since \reff{eq:SOS:Qdual} is a relaxation of \reff{eq:SOS:dual},
we know $(\gamma^*,y^*)$ must also be a maximizer of (\ref{eq:SOS:dual}).
\end{proof}

If $\mathscr{R}_d(S)$ is a closed cone,
then we only need to check $y^* \in \mathscr{R}_d(S)$ in the above.
Interestingly, when $S$ is compact, the moment cone $\mathscr{R}_d(S)$ is closed
\cite{LasserreBook2009,LaurentSOSmom2009,JNieLinearOptMoment}.
As introduced in the Subsection~\ref{ssc:LoptMom},
the membership $y^* \in \mathscr{R}_d(S)$ can be checked
by solving a truncated moment problem.
This can be done by solving the optimization \reff{eq:ATMPrelx}
for a generically selected objective.
Once $(\gamma^*,y^*)$ is confirmed to be a maximizer of (\ref{eq:SOS:dual}),
we show how to get a minimizer for \reff{md:polyDRPO}.
This is shown as follows.

\begin{theorem} \label{thm:x*:solveDRO}
Assume \reff{closure=cap} holds.
For a relaxation order $k$, suppose $(x^*,w^*)$ is a minimizer of \reff{eq:SOS:Qres}
and $(\gamma^*,y^*,z^*)$ is a maximizer of \reff{eq:SOS:Qdual}
such that $y^* \in \overline{ \mathscr{R}_d(S) }$.
Assume there is no duality gap between
\reff{eq:SOS:Qres} and \reff{eq:SOS:Qdual}, i.e., they have the same optimal value.
If the point $x^*$ belongs to the set $X$ and $f(x^*)= \langle f, w^* \rangle$,
then $x^*$ is a minimizer of \reff{minf0:Ax+b:in:K*}.
Moreover, if in addition the dual cone $K^*$ can be expressed as in \reff{def:K_y^*},
then $x^*$ is also a minimizer of \reff{md:polyDRPO}.
\end{theorem}
\begin{proof}
Let $f_1, f_2$ be optimal values of the optimization problems
\reff{eq:SOS:inirel} and \reff{eq:SOS:dual} respectively.
Then, by the weak duality, it holds that
\[
f_1 \ge f_2.
\]
The membership $y^* \in \overline{ \mathscr{R}_d(S) }$
implies that $(\gamma^*,y^*)$ is a maximizer of (\ref{eq:SOS:dual}),
by Theorem~\ref{thm:TightRel}. So $f_2 = \gamma^*-b^Ty^*$.
By the assumption, the primal-dual pair \reff{eq:SOS:Qres}-\reff{eq:SOS:Qdual}
have the same optimal value, so
\[
\langle f, w^*\rangle = \gamma^*-b^Ty^*  = f_2 .
\]
The constraint $h(x^*, \xi) \in \qmod{g}_{2k} + Y^*$
implies that $h(x^*, \xi) \in \mathscr{P}_d(S) + Y^*$.
Since $x^* \in X$, we know $x^*$ is a feasible point of \reff{minf0:Ax+b:in:K*}.
The optimal value of \reff{minf0:Ax+b:in:K*}
is greater than or equal to that of \reff{eq:SOS:inirel}, hence
\[
f_1 \ge f_2 = \langle f, w^*\rangle  = f(x^*) \ge f_1.
\]
So $f(x^*) = f_1$.
This implies that $x^*$ is a minimizer of \reff{minf0:Ax+b:in:K*}.
Moreover, if in addition $K^*$ can be expressed as in \reff{def:K_y^*},
the optimization \reff{md:polyDRPO} is equivalent to \reff{minf0:Ax+b:in:K*},
by Theorem~\ref{th:DRO<=>Cone}. So $x^*$ is also a minimizer of \reff{md:polyDRPO}.
\end{proof}

In the above theorem, the assumptions that $x^* \in X$ and
$f(x^*)= \langle f, w^* \rangle$ must hold
if $f,-c_1, \ldots, -c_{m_2}$ are SOS-convex polynomials.
We have the following theorem.

\begin{theorem} \label{thm:tight:SOScvx}
Assume \reff{closure=cap} holds.
For a relaxation order $k$, suppose $(x^*,w^*)$ is a minimizer of \reff{eq:SOS:Qres}
and $(\gamma^*,y^*,z^*)$ is a maximizer of \reff{eq:SOS:Qdual}
such that $y^* \in \overline{ \mathscr{R}_d(S) }$. Assume there is no duality gap between
\reff{eq:SOS:Qres} and \reff{eq:SOS:Qdual}, i.e., they have the same optimal value.
If $f,-c_1, \ldots, -c_{m_2}$ are SOS-convex polynomials,
then $x^*  \coloneqq  \pi(w^*)$  is a minimizer of \reff{minf0:Ax+b:in:K*}.
Moreover, if in addition $K^*$ can be expressed as in \reff{def:K_y^*},
then $x^*$ is also a minimizer of \reff{md:polyDRPO}.
\end{theorem}
\begin{proof}
Since $f$ and $-c_1, \ldots, -c_{m_2}$ are SOS-convex polynomials,
by the Jensen's inequality (see \cite{LasserreConvexity}), it holds that
\[
f(x^*) = f( \pi(w^*) ) \le \langle f, w^* \rangle,
\]
\[
c_i(x^*) = c_i( \pi(w^*) ) \ge \langle c_i, w^* \rangle, \, i=1,\ldots, m_2.
\]
Similarly, the constraint $L_{c_i}^{(d_1)}[w^*] \succeq 0$ implies that
$\langle c_i, w^* \rangle \ge 0$. So $x^* \in X$
is a feasible point of \reff{minf0:Ax+b:in:K*}.
As in the proof of Theorem~\ref{thm:x*:solveDRO}, we can similarly show that
\[
f_1 \ge f_2 = \langle f, w^* \rangle \ge f(x^*) \ge f_1 ,
\]
so $f(x^*)= \langle f, w^* \rangle$.
The conclusions follow from Theorem~\ref{thm:x*:solveDRO}.
\end{proof}

\subsection{An algorithm for solving the DROM}
\label{ssc:alg}

Based on the above discussions, we now give the algorithm
for solving the optimization problem \reff{eq:SOS:inirel}
and its dual \reff{eq:SOS:dual}, as well as the DROM (\ref{md:polyDRPO}).
\begin{alg}
\label{def:alg}
For given $f, h, \mc{M}, S, X, Y$
and the defining polynomial tuples $g$ and $c$, do the following:

\begin{itemize}
	
\item[Step 0] Get a computational representation for $\overline{ \cone{Y} }$
and the dual cone $Y^*$. Initialize
\[
d_0  \coloneqq  \lceil \deg(g)/2\rceil, \quad
t_0  \coloneqq  \lceil d/2 \rceil,\quad
k  \coloneqq  \lceil d/2 \rceil, \quad  l  \coloneqq  t_0+1.
\]
Choose a generic polynomial $R \in \Sig[\xi]_{2t_0+2}$.

\item[Step 1]
Solve \reff{eq:SOS:Qres} for a minimizer $(x^*, w^*)$
and solve \reff{eq:SOS:Qdual} for a maximizer $(\gamma^*,y^*,z^*)$.

\item[Step 2]
Solve the moment optimization
\be \label{TSMP:<R,v>}
%%	(V_{\hat{k}})\quad
\left\{
\baray{cl}
\min\limits_{\omega}  & \langle R, \omega\rangle\\
\st  & \omega|_d =y^* , \,
\omega \in\mathscr{S}[g]_{2\ell}, \, \omega \in \re^{ \N^p_{2\ell} }.
\earay
\right.
\ee
If \reff{TSMP:<R,v>} is infeasible, then $y^*$ admits no $S$-measure,
update $k \coloneqq k+1$ and go back to Step~1.
Otherwise, solve \reff{TSMP:<R,v>} for a minimizer $\omega^*$ and go to Step~3.
	
\item[Step 3] Check whether or not
there exists an integer $s \in [\max(d_0,t_0), \ell]$ such that
\[
 \rank\, M_{s-d_0}[\omega^*] \,= \, \rank\, M_{s}[\omega^*].
\]
If such $s$ does not exist, update $\ell \coloneqq \ell+1$ and go to Step~2.
If such $s$ exists, then $y^* = \int [\xi]_d \mt{d} \mu$ for the measure
\[
\mu \, = \,  \theta_1 \dt_{u_1}+\cdots+ \theta_r \dt_{u_r} .
\]
In the above, the scalars $\theta_1,\ldots,\theta_r>0$,
$u_1, \ldots, u_r \in S$ are distinct points, $r =\rank\, M_{s}[\omega^*]$,
and $\dt_{u_i}$ denotes the Dirac measure supported at $u_i$.
Up to scaling, a measure $\mu^*\in\mathcal{M}$
that achieves the worst case expectation constraint
can be recovered as a multiple of $\mu$.
%Output the atomic measures $\xi^{(1)},\ldots,\xi^{(r)}$
%together with the optimal value $F^*$ and the optimizer
%$x^*$ of (\ref{md:polyDRPO}), i.e.,
%\be  \label{eq:optimalsol}
% F^*= \langle f,w_k^*\rangle,\quad x^*=\pi(w_k^*).
%\ee
\end{itemize}
\end{alg}

\begin{remark}
All optimization problems in Algorithm~\ref{def:alg} can be solved numerically
by the software \texttt{GloptiPoly3} \cite{GloptiPoly3},
\texttt{YALMIP} \cite{LofbergYalmip} and \texttt{SeDuMi} \cite{JSturmSedumi}.
In Step 0, we assume $\overline{ \cone{Y} }$
can be expressed by linear, second order or semidefinite cones.
See Section~\ref{sec:numerical} for more details.
In Step~1, if \reff{eq:SOS:Qres} is unbounded from below,
then \reff{minf0:Ax+b:in:K*} must also be unbounded from below.
If \reff{eq:SOS:Qdual} is unbounded from above,
then \reff{eq:SOS:dual} may be unbounded from above
(and hence \reff{minf0:Ax+b:in:K*} is infeasible) ,
or it may be because the relaxation order $k$ is not large enough.
We refer to \cite{JNieLinearOptMoment} for how to
verify unboundedness of \reff{eq:SOS:dual}.
Generally, one can assume \reff{eq:SOS:Qres}
and \reff{eq:SOS:Qdual} have optimizers.
In Step~3, the finitely atomic measure $\mu$
can be obtained by computing Schur decompositions and eigenvalues.
We refer to \cite{ExtractSolGloptipoly} for the method.
It is also implemented in the software \texttt{GloptiPoly3}.
Note that the measure $\mu$ associated with $y^*$
may not belong to $\mathcal{M}$. This is because (\ref{eq:SOS:Qdual})
has the conic constraint $y\in \overline{cone(Y)}$ instead of $y\in Y$.
Once the atomic measure $\mu$ is extracted,
we can choose a scalar $\beta>0$ such that $\beta \mu\in \mathcal{M}$.
\end{remark}

\subsection{Convergence of Algorithm~\ref{def:alg}}
\label{ssc:convergence}

In this subsection, we prove the convergence of Algorithm \ref{def:alg}.
The main results here are based on the work
\cite{JNieAtruncated,JNieLinearOptMoment}.

First, we consider the relatively simple but still interesting case that
$\xi$ is a univariate random variable (i.e., $p=1$)
and the support set $S=[a_1, a_2]$ is an interval.
For this case, Algorithm~\ref{def:alg} must terminate in the initial loop
$k  \coloneqq  \lceil d/2 \rceil$ with $y^* \in \mathscr{R}_d(S)$,
if $(\gamma^*, y^*, z^*)$ is a maximizer of \reff{eq:SOS:Qdual}.

\begin{theorem} \label{thm:UTMPconv}
Suppose the random variable $\xi$ is univariate
and the set $S=[a_1,a_2]$, for scalars $a_1 < a_2$, is an interval
with the constraint $g(\xi)  \coloneqq  (\xi-a_1)(a_2-\xi) \ge 0$.
If $(\gamma^*, y^*, z^*)$ is a maximizer of \reff{eq:SOS:Qdual}
for $k = \lceil d/2 \rceil$,
then we must have $z^* \in \mathscr{R}_{2k}(S)$
and hence $y^* \in \mathscr{R}_d(S)$.
\end{theorem}
\begin{proof}
In the relaxation~\reff{eq:SOS:Qdual}, the tms $z$ has the even degree $2k$.
We label the entries of $z$ as
$
z = (z_0, z_1, \ldots, z_{2k} ).
$
The condition $z \in \mathscr{S}[g]_{2k}$ implies that
\begin{equation}
\label{eq:UTMPcondi}
M_k[z] \succeq 0, \quad L_{g}^{(k)}[z] \succeq 0.
\end{equation}
Since $g = (\xi-a_1)(a_2-\xi)$, one can verify that
$L_{g}^{(k)}[z] \succeq 0$ is equivalent to
\[
(a_1+a_2)\left[\baray{llcl}
z_1 & z_2 & \cdots & z_k \\
z_2 & z_3 & \cdots & z_{k+1} \\
\vdots & \vdots & \ddots & \vdots \\
z_k & z_{k+1} & \cdots & z_{2k-1} \\
\earay\right] \succeq
a_1a_2 \left[\baray{llcl}
z_0 & z_1 & \cdots & z_{k-1} \\
z_1 & z_2 & \cdots & z_{k} \\
\vdots & \vdots & \ddots & \vdots \\
z_{k-1} & z_{k} & \cdots & z_{2k-2} \\
\earay\right] +
\]
\[
\left[\baray{llcl}
z_2 & z_3 & \cdots & z_{k} \\
z_3 & z_4 & \cdots & z_{k+1} \\
\vdots & \vdots & \ddots & \vdots \\
z_{k} & z_{k+1} & \cdots & z_{2k} \\
\earay\right]  .
\]
As shown in \cite{CurtoUTMP,KreinUTMP}, the (\ref{eq:UTMPcondi})
are sufficient and necessary conditions for $z \in \mathscr{R}_{2k}(S)$.
So, if $(\gamma^*, y^*, z^*)$ is a maximizer of \reff{eq:SOS:Qdual},
then $M_k[z^*]\succeq 0$ and $L_{g}^{(k)}[z^*] \succeq 0$.
Hence, we have $z^* \in \mathscr{R}_{2k}(S)$ and hence
$y^* = z^*|_d \in \mathscr{R}_d(S)$.
\end{proof}

Second, we prove the asymptotic convergence of Algorithm~\ref{def:alg}
when the random variable $\xi$ is multi-variate.
It requires that the quadratic module $\qmod{g}$ is archimedean
and \reff{eq:SOS:inirel} has interior points.

\begin{theorem} \label{thm:asymptotic}
Assume that $\qmod{g}$ is archimedean
and there exists a point $\hat{x} \in X$ such that
$h(\hat{x}, \xi) = a_1(\xi) + a_2(\xi)$
with $a_1 > 0$ on $S$ and $a_2 \in Y^*$.
Suppose $(\gamma^{(k)}, y^{(k)},z^{(k)})$ is an optimal triple
of \reff{eq:SOS:Qdual} when its relaxation order is $k$.
Then, the sequence $\{ y^{(k)} \}_{k=1}^\infty$ is bounded
and every accumulation point of $\{ y^{(k)} \}_{k=1}^\infty$
belongs to the cone $\mathscr{R}_d(S)$.
Therefore, every accumulation point of
$\{ (\gamma^{(k)}, y^{(k)}) \}_{k=1}^\infty$
is a maximizer of (\ref{eq:SOS:dual}).
\end{theorem}
\begin{proof}
For every $(\gamma, y, z)$ that is feasible for \reff{eq:SOS:Qdual}
and for $\hat{w}  \coloneqq  [\hat{x}]_{2d_1}$, it holds that
\be \label{gap>=Ax+b*y}
 \langle f, \hat{w} \rangle - \big( \gamma - \langle b, y \rangle \big) =
\langle f-y^TAx-\gamma, \hat{w} \rangle + (A\hat{x}+b)^T y \ge  (A\hat{x}+b)^T y.
\ee
There exists $\eps>0$ such that $a_1(\xi) - \eps \in \qmod{g}_{2k_0}$,
for some $k_0 \in \N$, since $\qmod{g}$ is archimedean.
Noting $a_2 \in Y^*$, one can see that
\[
 (A\hat{x} +b)^T y = \langle h(\hat{x}, \xi), y \rangle =
\langle a_1(\xi), y \rangle + \langle a_2(\xi), y \rangle
\ge \langle a_1(\xi), y \rangle .
\]
For all $k \ge k_0$, it holds that
\[
\langle a_1(\xi), y \rangle = \langle a_1(\xi) -\eps, y \rangle
+ \eps \langle 1, y \rangle \ge \eps \langle 1, y \rangle  = \eps y_0.
\]
(Note $\langle 1, y \rangle  = y_0$.)
Let $f_2$ be the optimal value of \reff{eq:SOS:dual}, then
\[
\gamma^{(k)} - \langle b, y^{(k)} \rangle \ge f_2 ,
\]
because $(\gamma^{(k)}, y^{(k)},z^{(k)})$ is an optimizer of \reff{eq:SOS:Qdual},
and \reff{eq:SOS:Qdual} is a relaxation of the maximization \reff{eq:SOS:dual}.
So \reff{gap>=Ax+b*y} implies that
\[
 (A\hat{x}+b)^T y^{(k)} \le  \langle f, \hat{w} \rangle - f_2.
\]
Hence, we can get that
\[
(y^{(k)})_0 \le \frac{1}{\eps}(  \langle f, \hat{w} \rangle - f_2 ) .
\]
The sequence $\big\{ (y^{(k)})_0 \big\}_{k=1}^\infty$ is bounded.

Since $\qmod{g}$ is archimedean, there exists $N>0$ such that
$N - \| \xi \|^2 \in \qmod{g}_{2k_1}$ for some $k_1 \ge k_0$.
For all $k \ge k_1$, the membership $z^{(k)} \in \mathscr{S}[g]_{2k}$ implies that
\[
N \cdot (z^{(k)})_0 - \big( (z^{(k)})_{2e_1} + \cdots + (z^{(k)})_{2e_p} \big) \ge 0.
\]
Note that $y^{(k)}=z^{(k)}|_d$, hence $(y^{(k)})_0 = (z^{(k)})_0$.
Since $z^{(k)} \in \mathscr{S}[g]_{2k}$ and the sequence
$\big\{ (z^{(k)})_0 \big\}_{k=1}^\infty$ is bounded,
one can further show that the set
\[ \{z^{(k)}|_d : z \in \mathscr{S}[g]_{2k} \}_{k=1}^\infty \]
is bounded. We refer to \cite[Theorem~4.3]{JNieLinearOptMoment}
for more details about the proof.
Therefore, the sequence $\{ y^{(k)} \}_{k=1}^\infty$ is bounded.
Since $\qmod{g}$ is archimedean, we also have
\[
\mathscr{R}_d(S) \, = \, \bigcap_{k=1}^\infty  S_k, \quad \mbox{where} \quad
S_k  \coloneqq  \{z|_d: \, z \in \mathscr{S}[g]_{2k} \}.
\]
This is shown in Proposition~3.3 of \cite{JNieLinearOptMoment}.
So, if $\hat{y}$ is an accumulation point of $\{ y^{(k)} \}_{k=1}^\infty$,
then we must have $\hat{y} \in \mathscr{R}_d(S)$.
Similarly, if $(\hat{\gamma}, \hat{y}, \hat{z})$ is an
accumulation point of $\{ (\gamma^{(k)}, y^{(k)},z^{(k)}) \}_{k=1}^\infty$,
then $\hat{y} \in \mathscr{R}_d(S)$.
As in the proof of Theorem~\ref{thm:TightRel},
one can similarly show that $(\hat{\gamma}, \hat{y})$
is a maximizer of (\ref{eq:SOS:dual}). 
\end{proof}

Last, we prove that Algorithm~\ref{def:alg} will terminate within
finitely many steps under certain assumptions.
Like Theorem~\ref{thm:asymptotic},
we also assume the archimedeanness of $\qmod{g}$.
When $\qmod{g}$ is not archimedean, if the set
$S = \{\xi\in\re^p: g(\xi)\ge 0\}$ is bounded,
we can replace $g$ by $\tilde{g} = (g, N-\|\xi\|^2)$
where $N$ is such that $S \subseteq \{ \| \xi \|^2 \le N\}$.
Then $\qmod{\tilde{g}}$ is archimedean.
Moreover, we also need to assume the strong duality
between (\ref{eq:SOS:inirel}) and (\ref{eq:SOS:dual}),
which is guaranteed under the Slater's condition for (\ref{eq:SOS:dual}).
These assumptions typically hold for polynomial optimization.

\begin{theorem} \label{thm:finite}
Assume $\qmod{g}$ is archimedean and there is no duality gap
between \reff{eq:SOS:inirel} and \reff{eq:SOS:dual}.
Suppose $(x^*, w^*)$ is a minimizer of \reff{eq:SOS:inirel}
and $(\gamma^*, y^*)$ is a maximizer of \reff{eq:SOS:dual} satisfying:
\begin{itemize}
\item[(i)] There exists $k_1\in\mathbb{N}$ such that
$h(x^*, \xi) = h_1(\xi) + h_2(\xi)$,
with $h_1\in \qmod{g}_{2k_1}$ and $h_2 \in Y^*$.

\item[(ii)] The polynomial optimization problem in $\xi$
\be \label{POP:h1(xi)}
\left\{ \baray{cl}
\min\limits_{\xi \in \re^p} & h_1(\xi)   \\
 \st & g_1(\xi) \ge 0, \ldots, g_{m_1}(\xi) \ge 0
\earay  \right.
\ee
has finitely many critical points $u$ such that $h_1(u) = 0$.
\end{itemize}
Then, when $k$ is large enough, for every optimizer
$(\gamma^{(k)}, y^{(k)}, z^{(k)})$ of \reff{eq:SOS:Qdual},
we must have $y^{(k)} \in \mathscr{R}_d(S)$.
\end{theorem}
\begin{proof}
Since there is no duality gap between \reff{eq:SOS:inirel}
and \reff{eq:SOS:dual},
\[
0 = \langle f, w^* \rangle - \big( \gamma^* - \langle b, y^* \rangle \big) =
\langle f-(y^*)^TAx-\gamma^*, w^* \rangle + (Ax^*+b)^T y^*.
\]
Due to the feasibility constraints, we further have
\[
\langle f(x)-(y^*)^TAx-\gamma^*, w^* \rangle = 0, \quad  (Ax^*+b)^T y^* = 0.
\]
Therefore, it holds that
\[
(Ax^*+b)^T y^* = \langle h(x^*, \xi), y^* \rangle =
\langle h_1(\xi), y^* \rangle + \langle h_2(\xi), y^* \rangle =   0.
\]
The conic membership $y^* \in \overline{K}$ implies that
\[
\langle h_1(\xi), y^* \rangle = \langle h_2(\xi), y^* \rangle = 0.
\]
We consider the polynomial optimization problem
\reff{POP:h1(xi)} in the variable $\xi$.
For each order $k \ge k_1$, the $k$th order Moment-SOS relaxation pair
for solving \reff{POP:h1(xi)} is
\be \label{mom:min<h1,z>}
\min  \quad \langle h_1(\xi), z \rangle \quad \st \quad
z \in \mathscr{S}[g]_{2k} , z_{0} = 1, \
\ee
\be \label{sos:max:gm}
\nu_k  \coloneqq  \,\, \max  \quad \gamma \quad \st \quad
h_1(\xi) -\gamma \in  \qmod{g}_{2k}.
\ee
The archimedeanness of $\qmod{g}$ implies that $S$ is compact,
so \[ \overline{ \mathscr{R}_d(S) } = \mathscr{R}_d(S). \]
The membership $y^* \in \overline{K}$ implies that $y^* \in \mathscr{R}_d(S)$.
Since \[ \langle h_1(\xi), y^* \rangle=0, \]
the polynomial $h_1(\xi)$ vanishes on the support of
each $S$-representing measure for $y^*$,
so the optimal value of \reff{POP:h1(xi)} is zero.
By the given assumption, the sequence $\{\nu_k\}$
has finite convergence to the optimal value $0$
and the relaxation \reff{sos:max:gm} achieves its optimal value for all $k \geq k_1$.
The optimization \reff{POP:h1(xi)} has only finitely many critical points
that are global optimizers.
So, Assumption~2.1 of \cite{JNieFlatTruncation} for the optimization \reff{POP:h1(xi)}
is satisfied. Moreover, the given assumption also implies that
$(x^*, w^*)$ is an optimizer of \reff{eq:SOS:Qres} and
$(\gamma^*, y^*, z^*)$ is an optimizer of \reff{eq:SOS:Qdual} for all $k \ge k_1$.
Suppose $(x^{(k)}, w^{(k)})$ is an arbitrary optimizer of \reff{eq:SOS:Qres} and
$(\gamma^{(k)}, y^{(k)}, z^{(k)})$ is an arbitrary optimizer
of \reff{eq:SOS:Qdual}, for the relaxation order $k$.

When $(z^{(k)})_{0}=0$, we have $vec(1)^T M_k [z^{(k)}] vec(1) =0$.
Since $M_k[z^{(k)}] \succeq 0$,
\[ M_k[z^{(k)}] vec(1) =0 . \]
Consequently, we further have $M_k[z^{(k)}] vec(\xi^\af)=0$ for all $|\af| \leq k-1$
(see Lemma~5.7 of \cite{LaurentSOSmom2009}).
Then, for each power $\af = \bt + \eta$ with $|\bt|,|\eta| \leq k-1$,
one can get
$
(z^{(k)})_\af= vec(\xi^\bt)^T M_k[z^{(k)}] vec(\xi^\eta) =0.
$
This means that $z^{(k)}|_{2k-2}$ is the zero vector
and hence $y^{(k)} \in \mathscr{R}_d(S)$.

For the case $(z^{(k)})_{0}>0$, let $\hat{z}  \coloneqq   z^{(k)}/(z^{(k)})_0$.
The given assumption implies that $(x^*, w^*)$
is also a minimizer of \reff{eq:SOS:Qres} and $(\gamma^*, y^*,z^*)$
is optimal for \reff{eq:SOS:Qdual}, for all $k \ge k_1$.
So there is no duality gap between \reff{eq:SOS:Qres} and \reff{eq:SOS:Qdual}.
Since $(\gamma^{(k)}, y^{(k)}, z^{(k)})$ is optimal for \reff{eq:SOS:Qdual},
so $\langle h_1(\xi), z^{(k)} \rangle = 0$ and hence
$\hat{z}$ is a minimizer of \reff{mom:min<h1,z>} for all $k\geq k_1$.
By Theorem~2.2 of \cite{JNieAtruncated}, the minimizer
$z^{(k)}$ must have a flat truncation $z^{(k)}|_{2t}$ for some $t$,
when $k$ is sufficiently big.
This means that the truncation $z^{(k)}|_{2t}$, as well as $y^{(k)}$,
has a representing measure supported in $S$.
Therefore, we have $y^{(k)} \in \mathscr{R}_d(S)$.
\end{proof}

The conclusion of Theorem~\ref{thm:finite} is guaranteed to hold
under conditions (i) and (ii),
which depend on the constraints $g$ and the set $Y$.
These two conditions are not convenient to verify computationally.
However, in computational practice of Algorithm~\ref{def:alg},
there is no need to check or verify them.
The correctness of computational results by Algorithm~\ref{def:alg}
does not depend on conditions (i) and (ii).
In other words, the conditions (i) and (ii) are sufficient
for Algorithm~\ref{def:alg} to have finite convergence,
but they may not be necessary.
It is possible that the finite convergence occurs even if some of them fail to hold.
In our numerical experiments, the finite convergence is always observed.
We also like to remark that the conditions (i) and (ii) generally hold,
which is a main topic of the work \cite{JNieOptCondition}.
In particular, when $h_1$ has generic coefficients,
the optimization \reff{POP:h1(xi)} has finitely many critical points
and so the condition (ii) holds.
This is shown in \cite{JNieFlatTruncation}.

\section{Numerical Experiments}
\label{sec:numerical}

In this section, we give numerical experiments for
Algorithm \ref{def:alg} to solve distributionally robust optimization problems.
The computation is implemented in MATLAB R2018a,
in a Laptop with CPU 8th Generation Intel® Core™ i5-8250U and RAM 16 GB.
The software \texttt{GloptiPoly3} \cite{GloptiPoly3}, \texttt{YALMIP} \cite{LofbergYalmip}
and \texttt{SeDuMi} \cite{JSturmSedumi} are used for the implementation.
For neatness of presentation, we only display four decimal digits.

To apply implement Algorithm \ref{def:alg}, we need a computational representation
for the cone $\overline{cone(Y)}$. For a given set $Y$,
it may be mathematically hard to get a computationally efficient
description for the closure of its conic hull.
However, in most applications, the set $Y$ is often convex
and there usually exist convenient representations for $\overline{cone(Y)}$.
For instance, the $\overline{cone(Y)}$ is often a polyhedra, second order,
or semidefinite cone, or a Cartesian product of them.
The following are some frequently appearing cases.

\bit

\item If $Y=\{y : T y + u  \ge 0\}$ is a nonempty polyhedron,
given by some matrix $T$ and vector $u$, then
\begin{equation}
\label{coneY:linear}
\overline{ cone(Y) }  %%=\bigcup_{\delta\ge 0}\delta Y
\, = \, \{y: Ty+s u \ge 0,\, s \in \re_+ \}.
\end{equation}
It is also a polyhedron and is closed.

\item Consider that $Y =\{y : \mc{A}(y) + B \succeq 0 \}$ is given by
a linear matrix inequality, for a homogeneous linear symmetric matrix valued
function $\mc{A}$ and a symmetric matrix $B$.
If $Y$ is nonempty and bounded, then
\begin{equation}
\label{coneY:LMI}
\overline{ cone(Y) }  \, = \,
\left\{y : \mc{A}(y) + s B \succeq 0,\, s \in \re_+ \right\}  .
\end{equation}
When $Y$ is unbounded, the $\cone{Y}$ may not be closed
and its closure $\overline{ \cone{Y} }$ may be tricky.
We refer to the work \cite{Net10} for such cases.
When $Y$ is given by second order conic conditions,
we can do similar things for obtaining $\overline{ \cone{Y} }$.

\eit

\begin{exm}
\label{exm:basicDROM}
Consider the DROM problem
\begin{equation}
\label{eq:basicDROM}
\left\{
\begin{aligned}
\min_{ x \in \re^4 }\quad & f(x)=-x_1-2x_2-x_3+2x_4\\
\st \quad
&\inf_{\mu\in\mathcal{M}}\mathbb{E}_{\mu}[h(x,\xi)]\ge 0,\\
& %%c(x)=Cx+d \coloneqq (x,1-\mathbf{1}_4^Tx)\ge 0,
x \ge 0, \, 1-e^Tx \ge 0,
\end{aligned}
\right.
\end{equation}
where (the random variable $\xi$ is univariate, i.e, $p=1$)
\begin{align*}
h(x,\xi)=(x_4-x_1-2)\xi^5+(x_4-1)\xi^4+(2x_1+x_2+x_4+1)\xi^3\\
+(2x_1-x_2+x_4-1)\xi^2+(2-x_2-x_3)\xi,
\end{align*}
\[
S = [0,3], \quad  g = 3\xi - \xi^2,
\]
\[
Y = \left\{\left.y = \bbm y_0 \\ y_1 \\ \vdots \\ y_5 \ebm \in  \re^6 \right|
\begin{array}{c}
1 \le y_0 \le y_1 \le y_2\le \\
\quad y_3 \le y_4\le y_5 \le 2
\end{array}
\right\}.
\] 
The $\overline{cone(Y)}$ is given as in (\ref{coneY:linear}).
The objective $f$ and constraints $c_1,c_2$ are all linear.
We start with $k=3$, 
and the Algorithm~\ref{def:alg} terminates in the initial loop.
The optimal value $F^*$ and the optimizer
$x^*$ for \reff{minf0:Ax+b:in:K*} are respectively
\[
F^* \approx -0.0326,\quad x^*\approx(0.6775,0.0000,0.0000,0.3225).
\]
The optimizer for (\ref{eq:SOS:Qdual}) is
\[
y^* \approx (0.9355,0.9355,0.9517,1.0163,1.2260,1.8710).
\]
The measure $\mu$ for achieving $y^* = \int [\xi]_5 \mt{d} \mu$
is supported at the points
\[
u_1 \approx 0.9913,\quad u_2 \approx 3.0000.
\]
By a proper scaling, we get the measure
$\mu^* = 0.9957 \delta_{u_1} + 0.0043 \delta_{u_2}$
that achieves the worst case expectation constraint.
\end{exm}

\begin{exm}
\label{ex:SOSuni}
Consider the DROM problem
\begin{equation}
\label{eq:SOSuni}
\left\{
\begin{aligned}
\min_{x\in\mathbb{R}^3}\quad & f(x)=(x_1-x_3+x_1x_3)^2+(2x_2+2x_1x_2-x_3^2)^2\\
\st\quad & \inf_{\mu\in\mathcal{M}}\mathbb{E}_{\mu}[h(x,\xi)]\ge 0,\\
& c_1(x)=1-x_1^2-x_2^2-x_3^2 \ge 0,\\
& c_2(x)=3x_3-x_1^2-2x_2^4\ge 0,\\
\end{aligned}
\right.
\end{equation}
where (the random variable $\xi$ is bivariate, i.e, $p=2$)
\begin{align*}
h(x,\xi) = (1-x_3)\xi_1^2\xi_2^2+(x_1-x_2+x_3-1)\xi_1\xi_2^2+  \\
(x_1+x_2+x_3+1)\xi_2^2+ (x_1-x_3)\xi_1^2-\xi_2,
\end{align*}
\[
S = \{\xi \in\mathbb{R}^2 :\, 1-\xi^T\xi\ge 0\},
\quad g  \coloneqq  1-\xi^T \xi,
\]
\[
Y=\left\{
y\in\mathbb{R}^{ \N^2_4 }
\left|\begin{array}{c}
y_{00} = 1,\,\, 0.1 \le y_{\af} \le 1 \, (0< |\af| \le 4)   \\
\begin{pmatrix}
y_{20} & y_{11}  & y_{30} & y_{12} \\
y_{11} & y_{02}  & y_{21} & y_{03} \\
y_{30} & y_{21}  & y_{40} & y_{22} \\
y_{12} & y_{03}  & y_{22} & y_{04} \\
\end{pmatrix}
\preceq  2 I_4
\end{array}
\right.
\right\}.
\]
%The following is the {\tt YALMIP} code for implementing Algorithm \ref{def:alg}.
The $\overline{cone(Y)}$ is given as in (\ref{coneY:LMI}).
One can verify that $f$ and all $-c_i$ are SOS-convex.
We start with $k=2$, 
and Algorithm \ref{def:alg} terminates in the initial loop.
The optimal value $F^*$ and optimizer $x^*$ 
of \reff{minf0:Ax+b:in:K*} are respectively
\[
F^* \approx 0.0160,\quad x^*\approx(0.4060,0.0800,0.4706).
\]
The optimizer for \reff{eq:SOS:Qdual} is
\begin{align*}
y^* \approx (0.3180,0.2750,0.1411,0.2436,0.1137,0.0744,0.2199, 0.0950, \\
0.0552,0.0460,0.2011, 0.0819,0.0426,0.0318,0.0318).
\end{align*}
The measure $\mu$ for achieving $y^* = \int [\xi]_4 \mt{d} \mu$
is supported at the points
\[
u_1 \approx (0.6325,0.7745),\quad u_2 \approx (0.9434,0.3317).
\]
By a proper scaling, we get the measure
$\mu^* = 0.2527 \delta_{u_1}+ 0.7473 \delta_{u_2}$
that achieves the worst case expectation constraint.
\end{exm}

\begin{exm}
\label{ex:nonSOSuni}
Consider the DROM problem
\begin{equation}
\label{eq:nonSOSuni}
\left\{
\begin{aligned}
\min_{x\in\mathbb{R}^3}\quad & f(x)=x_1^4-2x_1^2+2x_2^3+x_3^4\\
\st\quad & \inf_{\mu\in\mathcal{M}}\mathbb{E}_{\mu}[h(x,\xi)]\ge 0,\\
& c_1(x)=x_1^2+x_2^2+x_3^2-1\ge 0,\\
& c_2(x) = 4-x_1^2-2x_2^2-x_3\ge 0,
\end{aligned}
\right.
\end{equation}
where (the random variable $\xi$ is bivariate, i.e, $p=2$)
\[
\baray{r}
h(x,\xi) = (x_1+x_2+1)\xi_2^4+(3x_1+x_2)\xi_1^2\xi_2+
(x_1+2x_2+x_3+1)\xi_1^3 \\ +2x_1+x_2-2x_3 ,
\earay
\]
\[
S \,= \, \{\xi\in\mathbb{R}^2 :\,
g  \coloneqq  (\xi_1,\xi_2,1-e^T\xi)\ge 0\},
\]
\[
Y = \left\{
y\in\mathbb{R}^{\mathbb{N}_4^2}\left|
\begin{array}{c}
y_{00}=1,\,
0.2^i\le y_{i0}\le 0.6^i, \\
y_{i0} \ge 1.2y_{0i},\, i=1,2,3,4
\end{array}
\right.
\right\}.
\]
In the above, $\overline{cone(Y)}$ is given as in (\ref{coneY:linear}).
The objective $f$ and $-c_1$ are not convex. We start with $k=2$, while the algorithm terminates at $k=3$. In the last loop, the optimizers for (\ref{eq:SOS:Qres}) and (\ref{eq:SOS:Qdual}) are
\[
\baray{r}
w^* \approx (1.0000,0.2692,-1.5454,-0.8493,0.0725,
-0.4161,-0.2287,2.3884,1.3125,\\
0.7213,0.0195,-0.1120,-0.0616,0.6430,0.3534,
0.1942,-3.6911,-2.0284,\\
   -1.1147,-0.6126,0.0053,-0.0302,-0.0166,0.1731,0.0951,
   0.0523,-0.9938,\\
   -0.5461,-0.3001,-0.1649,5.7044,3.1348,1.7227,0.9467,0.5202)\\
y^* \approx
(0.0871,0.0488,0.0383,0.0300,0.0188,
0.0195,0.0184,0.0116,0.0073,0.0122,\\
0.0113,0.0071,0.0045,0.0028,0.0094).
\earay
\]
The optimal value $F^* \approx -7.0017$ for both of them.
The measure for achieving $y^* = \int [\xi]_4 \mt{d} \mu$
is supported at the points
\[
u_1\approx(0.0000,1.0000),\quad u_2\approx (0.6139,0.3861).
\]
By a proper scaling, we get the measure
$\mu^* = 0.0877\delta_{u_1}+0.9123\delta_{u_2}$
that achieves the worst case expectation constraint. The point
\[
x^* = \pi(w^*)\approx(0.2692,-1.5454,-0.8493),
\]
is feasible for (\ref{eq:nonSOSuni}) as $c_1(x^*)\approx2.1822$
and $c_2(x^*)\approx3.9919\cdot 10^{-8}$.
Moreover, $F^*-f(x^*)\approx 1.2204\cdot 10^{-7}.$
By Theorem \ref{thm:x*:solveDRO}, we know $F^*$ is the optimal value and
$x^*$ is an optimizer for (\ref{eq:nonSOSuni}).
\end{exm}

\begin{exm} \label{ex:nonSOSadv}
Consider the DROM problem
\begin{equation}
\label{eq:nonSOSadv}
\left\{
\begin{aligned}
\min_{x\in\mathbb{R}^3}\quad & f(x)=x_1^4-x_1x_2x_3+x_3^3+3x_1x_3+x_2^2\\
\st\quad & \inf_{\mu\in\mathcal{M}}\mathbb{E}_{\mu}[h(x,\xi)]\ge 0,\\
& c_1(x)=x_1x_2-0.25\ge 0, \\
& c_2(x)=6-x_1^2-4x_1x_2-x_2^2-x_3^2\ge 0, 
\end{aligned}
\right.
\end{equation}
where (the random variable $\xi$ is bivariate, i.e, $p=2$)
\begin{align*}
h(x,\xi) = (2-x_1+x_2)\xi_2^4+(x_1+x_3+1)\xi_1\xi_2^2+(2-x_1+2x_2)\xi_2^3\\
+(x_1+2x_2+x_3+2)\xi_1^2+(3x_2-x_1)\xi_2^2,
\end{align*}
\[
S = \{\xi\in\mathbb{R}^2|1\le \xi^T\xi\le 4\},\quad
g = (\xi^T \xi -1, \, 4 - \xi^T\xi ),
\]
\[
Y = \left\{
y\in\mathbb{R}^{\mathbb{N}_4^2}\left|
y_{00}=1,\,
\sum_{|\alpha|\ge 1}y_{\alpha}^2=36
\right.
\right\} .
\]
The set $Y$ is not convex. Its convex hull is
$\| y \| \le \sqrt{37}$ with $y_{00} = 1$. Hence,
\[
\overline{cone(Y)} = \left\{y\in\mathbb{R}^{\mathbb{N}_4^2}
\left|\,
\| y \|_2\le \sqrt{37} y_{00}
\right.
\right\}.
\]
The functions $f$ and $-c_1,-c_2$ are not convex. We start with $k=2$.
The optimizers for (\ref{eq:SOS:Qres}) and (\ref{eq:SOS:Qdual}) are respectively
\[
\baray{r}
w^* \approx (
1.0000,0.6790,0.3682,-2.0984,0.4611
,0.2500,-1.4249,0.1356,-0.7726,\\
4.4034,0.3131,0.1698,-0.9675,0.0920
,-0.5246,2.9900,0.0499,-0.2845,\\
1.6212,-9.2402,0.2126,0.1153,-0.6569
,0.0625,-0.3562,2.0302,0.0339,\\
-0.1932,1.1008,-6.2742,0.0184,-0.1047
,0.5969,-3.4021,19.3898),
\\
y^* \approx
(1.2272,0.2992,-1.1902,0.0730,-0.2902
,1.1543,0.0178,-0.0708,0.2814,\\
-1.1194,0.0043,-0.0173,0.0686
,-0.2729,1.0857).
    \earay
\]
The optimal value is $F^* \approx -12.6420$ for both of them.
The measure for achieving $y^* = \int  [\xi]_4 \mt{d} \mu$ is
$\mu = 1.2272 \dt_{ u }$, with $u \approx (0.2438,-0.9698)\in S$.
So $\mu^*=\delta_u$.
For the point \[ x^* = \pi(w^*)\approx(0.6790,0.3682,-2.0984) , \]
one can verify that $x^*$ is feasible for (\ref{eq:nonSOSadv}), since
\[
c_1(x^*)\approx -1.6654\cdot10^{-9}, \,
c_2(x^*)\approx5.6235\cdot10^{-8}, \,
F^*-f(x^*)\approx  -7.7271\cdot10^{-8}.
\]
By Theorem \ref{thm:x*:solveDRO}, we know
$x^*$ is the optimizer for (\ref{eq:nonSOSadv}).
\end{exm}

\begin{exm} \label{exm:Portfolio}
(Portfolio selection \cite{DelageYeDRO,KlerkDROpolydense})
Consider that there exist $n$ risky assets that can be chosen
by the investor in the financial market.
The uncertain loss $r_i$ of each asset can be described
by the random risk variable $\xi$
which admits a probability measure supported in $S=[0,1]^p$.
Assume the moments of $\mu \in \mc{M}$ are constrained in the set
\[
Y=\left\{
y\in\mathbb{R}^{\mathbb{N}_3^3}\left|\,
y_{000}=1,\, 0.1\le y_{\alpha}\le 1,\, |\alpha|\ge 1
\right.
\right\} .
\]
The cone $\overline{cone(Y)}$ can be given as in (\ref{coneY:linear}).
Minimizing the portfolio loss over the ambiguity set $\mathcal{M}$
is equivalent to solving the following min-max optimization problem
\begin{equation}
\label{eq:portfolio_minmax}
\min_{x\in \Delta_3}\max_{\mu\in \mathcal{M}}\, \mathbb{E}_{\mu}\left[x_1r_1(\xi)+x_2r_2(\xi)+x_3r_3(\xi)\right],
\end{equation}
for the simplex
$\Delta_n \coloneqq \left\{x\in\mathbb{R}^3\left| e^Tx=1,\,x\ge 0\right.\right\}$.
The functions $r_i(\xi)$ are
\begin{equation}
\label{eq:def:r}
\left\{
\begin{aligned}
r_1(\xi) &= -1+\xi_1+\xi_1\xi_2-\xi_1\xi_3-2\xi_1^3,\\
r_2(\xi) &= -1-\xi_1\xi_2+\xi_2^2-\xi_2\xi_3+\xi_2^3,\\
r_3(\xi) &= -1+\xi_2\xi_3-\xi_3^2-\xi_3^3.
\end{aligned}
\right.
\end{equation}
Then (\ref{eq:portfolio_minmax}) can be equivalently reformulated as
\begin{equation}
\label{eq:portfolioDROM}
\left\{
\begin{array}{cl}
\min\limits_{(x_0,x)\in\mathbb{R}\times\mathbb{R}^3} & x_0\\
\st & \inf\limits_{\mu\in\mathcal{M}}\mathbb{E}_{\mu}\left[ x_0-\big(x_1r_1(\xi)+x_2r_2(\xi)+x_3r_3(\xi)\big)\right]\ge 0,\\
& x\ge 0,\, e^Tx=1.
\end{array}
\right.
\end{equation}
Applying Algorithm \ref{def:alg} to solve (\ref{eq:portfolioDROM}),
we get the optimal value $F^*$ and the optimizer $(x_0^*,x^*)$
in the initial loop $k=2$:
\[
F^*\approx -1.0136,\quad (x_0^*,x^*)\approx (-1.0136,0.1492,0.3501,0.5007).
\]
The optimizer for (\ref{eq:SOS:Qdual}) is
\[
\baray{r}
y^* \approx (
1.0000,0.6077,0.4440,0.3725,0.3864,0.3347,0.2530,0.4440,0.2666,
0.1803,\\
0.2560,0.2523,0.1771,0.3347,0.2010,0.1306,0.4440,0.2666,0.1601,
0.1000).
\earay
\]
The measure for achieving $y^*= \int [\xi]_4 \mt{d} \mu$ is
\[
\mu \,= \,  0.5560 \dt_{ u_1 } + 0.4440 \dt_{ u_2},
\]
with the following two points in $S$:
\begin{align*}
u_1 & \approx (0.4911,-0.0000,0.1905),\quad
u_2 \approx (0.7538,1.0000,0.6005).
\end{align*}
Since $\mu$ belongs to $\mathcal{M}$, it is also the measure
that achieves the worst case expectation constraint.
Therefore, the optimizer for \reff{eq:portfolio_minmax}
is $x^*$ and the optimal value is $-1.0136$.
\end{exm}

\begin{exm}[Newsvendor problem \cite{WiesemannConvexDRO}]
Consider that there is a newsvendor trade product
with an uncertain daily demand. Assume the demand quantity $D(\xi)$
is affected by a random variable $\xi\in\mathbb{R}^2$ such that
\[
D(\xi) \,=\, 2-\xi_1+\xi_2-\xi_1^2+2\xi_2^2+\xi_1^4.
\]
In each day, the newsvendor orders $x$ units of the product at the wholesale price $P_1$,
sells the product with quantity $\min\{x,D(\xi)\}$ at the retail price $P_2$
and clears the unsold stock at the salvage price $P_0$.
Assume that $P_0<P_1<P_2$, then the newsvendor's daily loss is given as
\[
l(x,\xi) \, \coloneqq  \, (P_1-P_2)x+(P_2-P_0) \cdot \max \{x-D(\xi),0\}.
\]
Clearly, the newsboy will earn the most if he can buy
the greatest order quantity that is guaranteed to be sold out.
Suppose $\xi$ admits a probability measure supported in $S$
and has its true distribution contained in the ambiguity set $\mathcal{M}$.
Then the best order decision for the newsvendor product
can be obtained from the following DROM problem
\begin{equation}
\label{eq:newsvendor}
\left\{
\begin{aligned}
\min_{x\in\mathbb{R}}\quad & (P_1-P_2)x\\
\st\quad & \inf_{\mu\in\mathcal{M}}\mathbb{E}_{\mu}[D(\xi)-x]\ge 0,\\
& x\ge 0.
\end{aligned}
\right.
\end{equation}
Suppose $P_0 = 0.25, P_1 = 0.5, P_2=1$, and
\[
S = [0,5]^2,\quad Y = \left\{y\in\mathbb{R}^{\mathbb{N}_4^2}\left|\,
 \begin{array}{c}
 y_{00}=1,\, 1\le y_{01}\le y_{02}\le 4\\
 2^i\le y_{i0}\le 4^i,\,i = 1, 2, 3, 4
 \end{array}
 \right.
 \right\}.
\]
The cone $\overline{cone(Y)}$ can be given as in (\ref{coneY:linear}).
Applying Algorithm \ref{def:alg} to solve (\ref{eq:newsvendor}),
we get the optimal value $F$ and the optimizer $x^*$ respectively
\[
F^* \approx -7.5000,\quad x^*\approx 15.0000.
\]
The optimizer of \reff{eq:SOS:Qdual} is
\[
\baray{r}
y^* \, \approx \,(0.5000,1.0000,0.5000,2.0000,1.0000,0.5000,4.0000,  2.0000, \\
 1.0000, 0.5000,8.0000,4.0000,2.0000,1.0000,0.5000).
\earay
\]
The measure for achieving $y^* = \int [\xi]_4 \mt{d} \mu$ is
$\mu = 0.5 \dt_{ u }$ with $u = (2.0000,1.0000) \in S$.
So $\mu^*=\delta_u$ achieves the worst case expectation constraint.
\end{exm}

We would like to remark that the ambiguity set $\mc{M}$
can be constructed by samples or historic data.
It can also be updated as the sampling size increases.
Assume the support set $S$ is given and each $\mu\in\mc{M}$ is a probability measure.
The moment ambiguity set $Y$ can be estimated by statistical samplings.
Suppose $T = \{\xi^{(1)},\ldots, \xi^{(N)}\}$ is a given sample set for $\xi$.
One can randomly choose $T_1,\ldots,T_s\subseteq T$
such that each $T_i$ contains $\lceil N/2\rceil$ samples.
Choose a smaller sample size $s$, say, $s = 5$.
For a given degree $d$, choose the moment vectors
$l,\, u\in\re^{\N_d^n}$ such that
\[
l_{\alpha} = \min\limits_{ j = 1,\ldots,s }
\Big\{\frac{1}{|T_j|}\sum\limits_{i\in T_j}(\xi^{(i)})^{\alpha},\,
\frac{1}{|T\setminus T_j|}\sum\limits_{i\in T\setminus T_j}(\xi^{(i)})^{\alpha}
\Big\},
\]
\[
u_{\alpha} = \max\limits_{ j = 1,\ldots,s }
\Big\{\frac{1}{|T_j|}\sum\limits_{i\in T_j}(\xi^{(i)})^{\alpha},\,
\frac{1}{|T\setminus T_j|}\sum\limits_{i\in T\setminus T_j}(\xi^{(i)})^{\alpha}
\Big\}
\]
for every power $\alpha\in\N_d^n$.
The moment constraining set $Y$,
e.g., as in Example~\ref{exm:Portfolio},
can be estimated as
\begin{equation}
\label{eq:Yapp}
Y  = \{ y\in\re^{\N_d^n}: l \le y\le u \}.
\end{equation}
Other types of moment constraining set $Y$ can be estimated similarly.
Suppose each $\xi^{(i)}$ independently follows the distribution of $\xi$.
As the sample size $N$ increases, the moment ambiguity set $\mc{M}$
with $Y$ in \reff{eq:Yapp} is expected to
give a better approximation of the true distribution of $\xi$.
This is indicated by the Law of Large Numbers
and the convergence results of sample average approximations.
The following is an example for how to do this.

\begin{example} \label{ex:extra}
Consider the portfolio selection optimization problem
as in Example~\ref{exm:Portfolio}.
The DROM is \reff{eq:portfolio_minmax},
or equivalently \reff{eq:portfolioDROM}.
Assume each $r_i(\xi)$ is given as in \reff{eq:def:r}.
Suppose $\xi = (\xi_1,\xi_2,\xi_3)$ is the random variable,
where each $\xi_i$ is independently distributed.
Assume $\xi_1$ follows the uniform distribution on $[0,1]$, $\xi_2$
follows the truncated standard normal distribution on $[0,1]$ and $\xi_3$
follows the truncated exponential distribution with the mean value $0.5$ on $[0,1]$.
We use the \texttt{MATLAB} commands \texttt{makedist} and \texttt{truncate}
to generate samples of $\xi$ with the sample size $N\in\{50,100,200\}$,
and then construct $Y$ as in \reff{eq:Yapp} with $s = 5$ and $d = n = 3$.

\noindent
(i). When $N = 50$, we get that
\[\begin{array}{r}
l = (1.0000,0.4354,0.3779,0.3873,0.2757,0.1916,0.1872,
    0.1975,0.1549,0.2018,\\
   0.2027,0.1299,0.1161,0.1111,0.0848,0.1025,0.1193
   ,0.0801,0.0866,0.1207),\\
u = (1.0000,0.5803,0.4606,0.4808,0.3938,
   0.2696,0.2579,0.2838,0.2109,0.3293,\\
   0.2913,0.1870,0.1821,0.1662,0.1091,0.1793,
   0.2027,0.1235,0.1361,0.2560).
\end{array}
\]

\noindent
(ii). For $N = 100$, we get that
\[\begin{array}{r}
l = (1.0000,0.4935,0.3799,0.4135,0.3150,0.1828,
   0.2065,0.1975,0.1745,0.2459,\\
   0.2261,0.1061,0.1268,0.0924,0.0837,0.1280,
   0.1195,0.0926,0.1102,0.1709),\\
u = (1.0000,0.5882,0.4545,0.5182,0.4156,0.2529,
   0.2838,0.2833,0.2294,0.3545,\\
   0.3178,0.1768,0.1941,0.1565,0.1242,0.1844,
   0.2035,0.1451,0.1570,0.2716).
\end{array}
\]

\noindent
(iii). For $N = 200$, we get that
\[\begin{array}{r}
l = (1.0000,0.4803,0.4177,0.4157,0.3170,
   0.1957,0.2253,0.2508,0.1784,0.2580,\\
   0.2310,0.1274,0.1459,0.1170,0.0875,
   0.1348,0.1719,0.0998,0.1048,0.1886),\\
u = (1.0000,0.5647,0.4698,0.5137,0.3939,
   0.2712,0.2738,0.2883,0.2250,0.3387,\\
   0.3097,0.1904,0.1950,0.1662,0.1300,
   0.1889,0.2062,0.1396,0.1510,0.2470).
\end{array}
\]
Applying Algorithm~\ref{def:alg}, we get the optimal value $F^*$
and the optimizer $(x_0^*,x^*)$ in the initial loop $k=2$ for each case.
The computational results are given in Table~\ref{tab:extra}.
Since $x_0^*=F^*$ and $y^*$ admits a measure $\mu = \theta_1\delta_{u_1}+\theta\delta_{u_2}+\theta_3\delta_{u_3}$, we only list $F^*$, $\theta = (\theta_1,\theta_2,\theta_3)$ and $u_1,u_2,u_3$ for convenience.
\begin{table}[htb!]
\caption{Computational results for Example~\ref{ex:extra}}
\label{tab:extra}
\begin{tabular}{|c|c|c|c|c|}
\hline
Case & $F^*$ & $x^*$ & $\theta$ & $u_1,\, u_2,\,u_3$\\
\hline
(i)  & $-0.9711$ & $\bbm 0.0159\\0.2824\\0.7017\ebm$ &
$\bbm0.0746\\0.7297\\0.1957\ebm$ & $\bbm 0.0000\\
0.4555\\0.1266\ebm,\bbm 0.5675\\0.0000\\0.1074\ebm,
\bbm 0.8492\\1.0000\\0.3832\ebm$ \\ \hline

(ii) & $-0.9743$ & $\bbm 0.0163\\0.2711\\0.7125\ebm$ &
$\bbm 0.0749\\0.7320\\0.1931\ebm$ &
	$\bbm 0.0000\\0.5170\\0.1563\ebm,
\bbm0.5898\\0.0000\\0.1103\ebm,\bbm 0.8103\\1.0000\\0.4343\ebm$ \\ \hline

(iii) & $-0.9749$ & $\bbm 0.0171\\0.2829\\0.6999\ebm$ &
$\bbm 0.0937\\0.7096\\0.1967\ebm$ &
$\bbm 0.0000\\0.4666\\0.1370\ebm,
\bbm 0.5616\\0.0000\\0.1122\ebm,
\bbm 0.8450\\1.0000\\0.4474\ebm$\\
\hline
\end{tabular}
\end{table}
As the sample size increases, the optimal value of $F^*$ improves.
This indicates that the ambiguity set
can be estimated by sampling averages
and the accuracy increases as the sampling size increases.
\end{example}

\section{Conclusions and discussions}
\label{sec:con}

This paper studies distributionally robust optimization when the ambiguity set
is given by moment constraints.
The DROM has a deterministic objective,
some constraints on the decision variable and a worst case expectation constraint.
The distributionally robust min-max optimization is a special case of DROM.
The objective and constraints are assumed to be
polynomial functions in the decision variable.
Under the SOS-convexity assumption, we show that the DROM
is equivalent to a linear conic optimization problem with moment constraints,
as well as the psd polynomial conic condition.
The Moment-SOS relaxation method (i.e., Algorithm~\ref{def:alg})
is proposed to solve the linear conic optimization.
The method can deal with moments of any order.
Moreover, it not only returns the optimal value and optimizers for the original DROM,
but also gives the measure that achieves the worst case expectation constraint.
Under some general assumptions (e.g., the archimedeanness),
we proved the asymptotic and finite convergence of the proposed method
(see Theorems~\ref{thm:UTMPconv}, \ref{thm:asymptotic} and \ref{thm:finite}).
Numerical examples, as well as some applications,
are given to show how it solves DROM problems.

The distributionally robust optimization is attracting broad interests
in various applications. There is much future work to do.
In this paper, we assumed the random function $h(x, \xi)$
is linear in the decision variable $x$.
How can we solve the DROM if $h(x, \xi)$ is not linear in $x$?
To prove the DROM \reff{md:polyDRPO} is equivalent to
the linear conic optimization \reff{eq:SOS:inirel},
we assumed the objective and constraints are SOS-convex.
When they are not SOS-convex,
how can we get equivalent linear conic optimization for \reff{md:polyDRPO}?
They are important future work.

\bigskip \noindent
{\bf Competing Interests} \,
The authors have no relevant financial or non-financial interests to disclose.

\bigskip \noindent
{\bf Author Contribution} \,
The authors have done analysis and computational results.

\medskip \noindent
{\bf Data Availability}
The paper does not analyse or generate any datasets, because the work proceeds
within a theoretical and mathematical approach.

\end{document}